\documentclass{amsart}
\usepackage{times}
\usepackage{amsfonts}
\usepackage{amsmath}
\usepackage{amscd}
\usepackage{amssymb}
\usepackage{amsxtra}
\usepackage{latexsym}
\usepackage{amsthm}
\usepackage{mathabx}
\usepackage[bookmarks,colorlinks,pagebackref]{hyperref} 
\input{xy}
\xyoption{all}

\theoremstyle{plain}

\swapnumbers
\newtheorem{theorem}[subsection]{Theorem}
\newcommand\Thm[1]{Theorem~\ref{#1}}
\newtheorem{corollary}[subsection]{Corollary}
\newcommand\Cor[1]{Corollary~\ref{#1}}
\newtheorem{lemma}[subsection]{Lemma}
\newcommand\Lem[1]{Lemma~\ref{#1}}
\newtheorem{proposition}[subsection]{Proposition}
\newcommand\Prop[1]{Proposition~\ref{#1}}

\newtheorem*{citedtheorem}{Theorem}

\theoremstyle{definition}

\newtheorem{example}[subsection]{Example}
\newcommand\Examp[1]{Example~\ref{#1}}

\theoremstyle{remark}

\newtheorem{remark}[subsection]{Remark}
\newcommand\Rem[1]{Remark~\ref{#1}}

\swapnumbers

\newcommand{\emptyprop}{q}

\newcommand \acf{algebraically closed field}

\newcommand \ann[2]{\operatorname{Ann}_{#1}(#2)}

\newcommand \CM{Coh\-en-Mac\-au\-lay}
\newcommand \complet[1]{\widehat {#1}}

\renewcommand \hom [3]{\operatorname{Hom}_{#1}(#2,#3)} 
\newcommand \homo{homomorphism}
\newcommand \id{\mathfrak a}
\renewcommand\iff{if and only if}
\newcommand\into{\hookrightarrow}

\newcommand \inverse[2]{{#1^{-1}(#2)}}
\newcommand \iso{\cong}
\newcommand \loc{{\mathcal {O}}}

\newcommand \map[1]{{\newcommand{\tmpprop}{#1q}  \if\tmpprop\emptyprop \to\else \xrightarrow{{\phantom{i}{#1}\phantom{i}}}\fi}} 
\newcommand \maxim{\mathfrak m}
\newcommand \nat{\mathbb N}

\newcommand \onto{\twoheadrightarrow}
\newcommand \op\operatorname
\newcommand \pol[2]{#1[#2]}
\newcommand \pow[2]{#1[[#2]]}

\newcommand \pr{\mathfrak p}
\newcommand \primary{\mathfrak g}
\newcommand \range [2]{#1,\dots,#2}
\newcommand \restrict [2]{\left.#1\right|_{{#2}}}

\let\sub\subseteq

\newcommand \tensor{\otimes}

\newcommand \zet{\mathbb Z}

\newcommand\of{\vee}

\newcommand \exactseq [5]{0\to{#1}\:\map{#2}\:{#3}\:\map{#4}\:{#5}\to0}
\newcommand \Exactseq [3]{0\to {#1}\to {#2}\to {#3}\to 0}

\newcommand{\commdiagram}[9][]{%
\begin{equation}
{\newcommand{\tmpprop}{#1q} 
\if\tmpprop\emptyprop \relax\else \label{#1}\fi}
\begin{aligned}%
\mbox{
\begin{picture}(130,90)%
\put(120,70){\vector( 0,-1){50}}%
\put(10,80){\vector( 1, 0){100}}%
\put(0,70){\vector( 0,-1){50}}%
\put(10,10){\vector( 1, 0){100}}%
\put(115,80){\makebox(0,0)[l]{$#4$}}%
\put(5,80){\makebox(0,0)[r]{$#2$}}%
\put(115,10){\makebox(0,0)[l]{$#9$}}%
\put(5,10){\makebox(0,0)[r]{$#7$}}%
\put(-3,50){\makebox(0,0)[r]{$#5$}}
\put(123,50){\makebox(0,0)[l]{$#6$}}
\put(60,3){\makebox(0,0)[c]{$#8$}}
\put(60,88){\makebox(0,0)[c]{$#3$}}
\end{picture}}
\end{aligned}
\end{equation}}

\newcommand\commtrianglefront[7][]{%
\begin{equation}
{\newcommand{\tmpprop}{#1q} 
\if\tmpprop\emptyprop \relax\else \label{#1}\fi}
\begin{aligned}%
\mbox{
\begin{picture}(120,80)%
\put(55,68){\vector(-1,-2){30}}
\put(65,68){\vector(1,-2){30}}
\put(30,5){\vector(1,0){60}}
\put(60,75){\makebox(0,0)[c]{$#2$}}
\put(25,5){\makebox(0,0)[r]{$#4$}}
\put(95,5){\makebox(0,0)[l]{$#6$}}
\put(60,0){\makebox(0,0)[c]{$#5$}}
\put(37,43){\makebox(0,0)[r]{$#3$}}
\put(83,43){\makebox(0,0)[l]{$#7$}}
\end{picture}}
\end{aligned}
\end{equation}}

\newcommand\commtriangleback[7][]{%
\begin{equation}
{\newcommand{\tmpprop}{#1q}
\if\tmpprop\emptyprop \relax\else \label{#1}\fi}
\begin{aligned}%
\mbox{
\begin{picture}(120,80)%
\put(55,70){\vector(-1,-2){30}}
\put(65,70){\vector(1,-2){30}}
\put(30,5){\vector(1,0){60}}
\put(60,75){\makebox(0,0)[c]{$#2$}}
\put(25,5){\makebox(0,0)[r]{$#6$}}
\put(95,5){\makebox(0,0)[l]{$#4$}}
\put(60,0){\makebox(0,0)[c]{$#5$}}
\put(37,43){\makebox(0,0)[r]{$#7$}}
\put(83,43){\makebox(0,0)[l]{$#3$}}
\end{picture}}
\end{aligned}
\end{equation}}

\newcommand\addgr[1]{\mathbf{G}^{\text{eq}}(#1)}
\newcommand\splus{\ssum}
\newcommand\Splus{\Ssum}
\newcommand\ordsum\oright
\newcommand\opord[1]{\op{\tt{#1}}}
\newcommand\hlc[3]{\op{H}_{#3}^{#2}(#1)}
\newcommand\lcl[1]{o_{#1}}

\newcommand\mul[2]{o_{#2}(#1)}
\newcommand\aleq{\preceq}
  \newcommand\add{equilateral}
  \newcommand\fl[1]{\mathfrak{#1}}
  
  \newcommand\ndo[1]{\op{End}(#1)}
    \newcommand\openndo[1]{\mathfrak O(#1)}

\newcommand\binord[1]{\opord o({#1})}
\newcommand\val[1]{\op{val}(#1)}

\newcommand\order[2]{\op{ord}_{#1}(#2)}

\newcommand\lc[2]{\op{H}^0_{#2}(#1)}

\newcommand\maxel{maximum}
\newcommand\minel{minimum}

\newcommand\ord{\opord {ORD}}

\newcommand\low[1]{\opord{ord}(#1)}

\newcommand\Ssum{\bigoplus}
\renewcommand\ssum{\oplus}
\newcommand\bsum{\mathbin{\bar\oplus}}
\newcommand\tsum{\mathbin{\tilde\oplus}}

\newcommand \len[1]{\opord{len}(#1)}

\newcommand \lenmod[2]{\opord{len}_{#1}(#2)}

\newcommand \fcyc[2]{\op{cyc}_{#1}(#2)}

\newcommand  \sch{Grassmanian}
\newcommand  \hdim{height rank}

\newcommand  \hd[1]{\opord h(#1)}
\newcommand  \cohrk[1]{\opord{cohrk}(#1)}
\newcommand  \odmod[2]{\opord o_{#1}(#2)}
\newcommand  \hdmod[2]{\opord h_{#1}(#2)}
\newcommand \gr[1]{\op{Grass}(#1)}
\newcommand \grass[2]{\op{Grass}_{#1}(#2)}
\newcommand \zeroid[1]{\mathbf 0_{#1}}

\newcommand \init[2]{{(#1,#2]}}
\newcommand \term[2]{{[#2,#1)}} 

\title {The theory of ordinal length}
\author{Hans Schoutens}
\date\today
\address{Department of Mathematics\\
365 5th Avenue\\
the CUNY Graduate Center\\
New York, NY 10016, USA}

\begin{document}
\begin{abstract} 
We generalize the notion of length to an ordinal-valued invariant defined on the class of finitely generated modules over a Noetherian ring. A key property of this invariant is its semi-additivity on short exact sequences. We show how to calculate this combinatorial invariant by means of the fundamental cycle of the  module, thus linking the lattice of submodules to homological properties of the module. Using this, we equip each module with its canonical topology.
\end{abstract}

\maketitle



\section{Introduction}
The \emph{length}   of an Artinian, finitely generated module $M$ is      defined as the longest chain of submodules in $M$. Since we have the descending chain condition, such a chain is finite, and hence can be viewed as a finite ordinal (recall that an \emph{ordinal} is a linearly ordered set with the descending chain condition). Hence we can immediately generalize this by transfinite induction to arbitrary Artinian modules, getting an ordinal-valued length function. To remain in the more familiar category of finitely generated modules, observe that at least over a complete Noetherian local ring, the latter is anti-equivalent with the class of Artinian modules via Matlis duality. We could have used this perspective (which we will discuss  in a future paper), but a moment's reflection directs us to a simpler solution: just reverse the order. Indeed, if we view the class of all submodules
 of a Noetherian module $M$, the \emph\sch\ $\grass RM$, as a partially ordered set by \emph{reverse}  inclusion, then $\grass RM$ admits the descending chain condition, and hence any subchain is well-ordered, that is to say, an ordinal. We then simply define $\len M$ as the supremum of all such chains/ordinals in $\grass RM$. Viewed as a module over itself, this yields the \emph{length} $\len R$ of a Noetherian ring $R$. In this paper, however, we will actually   define length through a foundation rank, and then show that it coincides with the above notion.

The key property of ordinary length is its additivity on short exact sequences. An example like $\exactseq\zet  2\zet{}{\zet/2\zet}$ immediately shows this can no longer hold in the general transfinite case. Moreover, even the formulation of additivity becomes problematic since ordinal sum, denoted in this paper as $\ordsum$, is not commutative. There does exist a different, commutative sum, $\ssum$, called in this paper the \emph{shuffle sum} (see Appendix~\ref{s:ssum}; for a brief introduction to ordinals, see \S\ref{s:Ord}). As our  first main result shows, both additions play a role:

\begin{citedtheorem}[Semi-additivity, \Thm{T:semadd}]
If $\Exactseq MNQ$ is an exact sequence of Noetherian modules, then $\len Q\ordsum\len N\leq \len M\leq \len Q\ssum\len N$.
\end{citedtheorem}

To appreciate the power of this result, notice that we instantaneously recover Vasconcelos' observation that a surjective endomorphism on a module is also injective (see \Cor{C:Vasc} below). In fact, we can prove the following generalization, which essentially says that endomorphisms cannot `expand':

\begin{citedtheorem}[Non-expansion, \Thm{T:endosupim}]
Let $f$ be an endomorphism on a Noetherian module $M$ and let $N$ be a submodule such that $N\sub f(N)$. Then $f(N)= N$ and the restriction $\restrict fN$ is an automorphism of $N$.
\end{citedtheorem}

We also get a new proof and a generalization of a result by Miyata  \cite{Miy}:  any exact sequence of the form  $M\to M\oplus N\to N\to 0$ must be split exact (this was proved independently  in  \cite{StrExt}).
As a last application, let us call a a non-unit $x\in R$  a \emph{parameter} if $\dim(R/xR)<\dim R$ (in the local case this is equivalent with $\dim{R/xR}=\dim R-1$, but in general, the dimension can drop more than one). 

\begin{citedtheorem}[Parameter Criterion, \Thm{T:lenpar}]
Let $R$ be a $d$-dimensional ring. An element $f\in R$ is a parameter \iff\ as a module, $\ann Rf$ has dimension at most  $d-1$.
\end{citedtheorem}

Although defined as a combinatorial invariant, length turns out to also encode some homological properties of a module. To formulate this, we must assume that the ring $R$, or at least the module $M$, has finite Krull dimension, an assumption we henceforth make. 
In intersection theory,  one associates to a module $M$  its \emph{fundamental cycle} $\fcyc {}M$ as the formal sum (Chow cycle)  $\sum\mul M\pr  [\pr]$, where the coefficient  $\mul M\pr  $ is the \emph{local multiplicity} of $M$ at $\pr$ (defined as the length of the zero-th local cohomology of $M_\pr$; see \S\ref{s:cohrk}). Length turns out to be an ordinal variant of this fundamental cycle:

\begin{citedtheorem}[Cohomological Rank, \Thm{T:cohrk}]
The length of a finite-dimensional Noetherian $R$-module $M$ is equal to the shuffle sum
$$
\len R=\Ssum_{\pr\in\op{Ass}(M)} \mul M\pr  \omega^{\dim{(R/\pr)}}.
$$
\end{citedtheorem}

It follows that the degree of the length of $M$ is equal to the dimension of  $M$. In particular,   a ring $R$ is   a domain \iff\ $\len R=\omega^d$, where $d$ is the dimension of $R$. In fact, the latter two results can also be proven by the theory of deviations and generalized Krull dimension initiated by Gabriel and Renschler (\cite{GabCat,GabRen}); see for instance, \cite[Proposition 6.1.10]{McCRob}.     However, our current theory is entirely distinct from this theory, as it applies only to Noetherian  modules: a module has ordinal length \iff\ it is Noetherian. On the other hand, ordinal length is a much finer invariant than Krull dimension. So, although ordinals have been used in the past by algebraists (\cite{GabRen,GarDec,NasOys,Sim}), it seems that our theory is their first foray into commutative algebra (perhaps with the exception of their short appearance, very much in the spirit of the current paper, in \cite{AschPong}). 
%
For the reader less adept at this concept from logic, the paper starts with a section on ordinals, and in an appendix, I explain   shuffle sums. 

The last two sections contain applications of the theory that do not mention length itself: acyclicity   (\S\ref{s:acyc}), and degradation (\S\ref{s:degrad}). These are merely meant   as an illustration of the power of the theory.  However, since any two domains of the same dimension have the same length (\Thm{T:dim}), as do and any two projective modules of the same rank (\Prop{P:lenproj}), length is not very sensitive to singularities nor to local/global phenomena. However, one can build ordinal-valued invariants from it that do measure these, such as the \emph{filtration rank}  of a module. This, and other  ordinal invariants, will be discussed in a series of future papers on applications of ordinal length to local cohomology, Koszul homology, \CM\ singularities,  prime filtrations, ordinal Hilbert and Poincare series, algebraic entropy, endomorphisms, \dots.
The preprints \cite{SchBinEndo} and \cite{SchCond}, on the other hand, study modules whose ordinal lengths have special properties.

\section{Notation and generalities on ordered sets}
  An \emph{ordered}
set $P$ (also called a \emph{partially ordered} set or \emph{poset}), is a set
together with a reflexive, antisymmetric and transitive binary relation
$\leq_P$, called the \emph{ordering} of $P$, and almost always written as 
$\leq$, without a subscript. Almost always, our posets will have  \emph{endpoints}, that is to say, a (unique)  least element $\bot$ and greatest element $\top$. 
A partial order is \emph{total} if for any two elements $a,b\in P$ either $a\leq b$ or $b\leq a$. A subset $C\sub P$ is called a \emph{chain}, if its induced order is total.
 If $a\leq b$, then
we may express
this by saying that \emph{$a$ is below $b$};
if $a<b$ (meaning that $a\leq b$ and $a\neq b$), we also say that \emph{$a$ is
strictly below $b$}.
More generally, for subsets
$A,B\sub P$, we say $A$ is \emph{below} $B$, 
and
write $A\leq B$, to mean that $a\leq b$ for
all $a\in A$ and all $b\in B$.

The   \emph{initial closed interval}  determined by $a\in P$ is by definition the set of 
$b\in  P$ with $b\leq a$ and will be denoted   $\init Pa$. Dually, the
 \emph{terminal closed interval} of $a$, denoted  $\term Pa$,
is the collection of all $b\in P$
with  $a\leq b$.

\subsection{Ordinals}\label{s:Ord}
A partial ordering is called a \emph{partial well-order} if it has the
descending chain condition, that is to say, any descending chain must
eventually be constant. A total order is a well-order \iff\ every non-empty
subset has a minimal element.  

Recall that an \emph{ordinal} $\opord a$ is an equivalence class, up to an order-preserving  isomorphism, of a total well-order (we will use a special font to distinguish ordinals from ordinary numbers).
  We say that $\opord a\leq \opord b$ if $\opord a$ is isomorphic to an initial segment of $\opord b$. This well-orders the class of all ordinals, and in particular, any \emph{set} of  ordinals has a minimum and a supremum. The finite ordinals are just the natural numbers (where we identify $n$ with the order $0<1<\dots<n-1$); the first infinite ordinal is the order-type of $(\nat,<)$ and is denoted $\omega$. Ordinal sum is defined by concatenation: $\opord a\ordsum\opord b$ is the ordinal obtained by putting $\opord a$ before $\opord b$ (see \S\ref{s:sum}). As this depends on the order, this sum is not commutative: $1\ordsum\omega\neq\omega\ordsum1$ since the former is just $\omega$. (Ordinal sum is usually denoted simply by $\opord a+\opord b$, but this might be misleading for algebraists, as this is not a commutative operation, and so our notation reflects that the right hand side is `dominant': smaller terms on the left are `gobbled up'.)  With a few minor exceptions, we will not use ordinal multiplication explicitly, and so $\omega^n$ will just mean the order-type of $\nat^n$ with its lexicographical ordering, and $a\omega^n$, for $a\in\nat$, will mean the $a$-fold sum of $\omega^n$ with itself (in textbooks, this would normally be denoted $\omega^na$, following Cantor's original notation, but as this  is quite awkward  for the algebraically inclined, we keep the more natural `scalar' multiplication notation).   
%
%

The supremum of all $\omega^n$ for $n\in\nat$ is denoted  $\omega^\omega$. All ordinals considered in this paper will be less than $\omega^\omega$. They therefore admit a unique  Cantor normal form  
\begin{equation}\label{eq:Cantor}
\opord a=a_d\omega^d\ordsum a_{d-1}\omega^{d-1}\ordsum \dots\ordsum a_1\omega\ordsum a_0
\end{equation} 
 with $a_i\in\nat$. We call the $a_i$ the \emph{Cantor coefficients} of $\opord a$, and denote them by  $\mul {\opord a}i:=a_i$. We call the least $i$ (respectively, the largest $i$) such that $\mul {\opord a}i\neq 0$ the \emph{order} $\order{}{\opord a}$ (respectively, the \emph{degree} $\deg{\opord a}$) of $\opord a$; the sum of all $\mul {\opord a}i$ is called its \emph{valence} $\val{\opord a}$.
An ordinal $\opord a$ is called a \emph{successor ordinal}  if it has an immediate predecessor.
This is equivalent with $\order{}{\opord a}=0$. The second addition, the shuffle sum $\ssum$, can be defined using Cantor normal forms as follows (see Appendix~\ref{s:ssum} for details): for each $i$ we have $\mul{\opord a\ssum\opord b}i=\mul{\opord a}i+\mul{\opord b}i$. Thus, for instance,  $(\omega^3\ordsum\omega)\ssum(\omega^2\ordsum2\omega)=\omega^3\ordsum\omega^2\ordsum3\omega$, which we will therefore also denote as $\omega^3\ssum\omega^2\ssum3\omega$. Note, however, that $(\omega^3\ordsum\omega)\ordsum(\omega^2\ordsum2\omega)=\omega^3\ssum\omega^2\ssum2\omega$, where the first $\omega$ gets `gobbled up' by the   $\omega^2$ to its right. An alternative way of viewing ordinals is as those surreal numbers (a la Conway, see, for instance \cite{KnuSur}) born last  on any given day; the addition in the field of surreals then corresponds to the shuffle sum $\ssum$. Taking the latter point of view also endows the ordinals with a commutative multiplication, but the only instance we need is multiplication with some power $\omega^n$ or some scalar $n\in\nat$. We have the obvious rule $\omega^n\cdot \omega^m=\omega^{m+n}$ and by `linearity', we extend this  for an arbitrary ordinal $\opord a$ with Cantor normal form \eqref{eq:Cantor} to
\begin{equation}\label{eq:multord}
\omega^n\cdot\opord a:=a_d\omega^{d+n}\ordsum a_{d-1}\omega^{d+n-1}\ordsum \dots\ordsum a_1\omega^{n+1}\ordsum a_0\omega^n.
\end{equation} 
Similarly, we define $n\cdot \opord a$ as the $n$-fold shuffle sum $\opord a\ssum\dots\ssum\opord a$, that is to say,
\begin{equation}\label{eq:scalord}
n\cdot\opord a:=na_d\omega^d\ordsum na_{d-1}\omega^{d-1}\ordsum \dots\ordsum na_1\omega^1\ordsum na_0.
\end{equation} 

Given any $e\geq 0$, we will write 
\begin{equation}\label{eq:splitord}
\opord a^{\geq e}:= \mul {\opord a}d\omega^d\ssum \dots\ssum \mul {\opord a}e\omega^e\qquad\text{and}\qquad \opord a^{\leq e}:=\mul {\opord a}e\omega^e\ssum \dots\ssum \mul {\opord a}0,
\end{equation} 
and a similar meaning for $\opord a^{> e}$ and $\opord a^{< e}$. 
In particular, $\opord a=\opord a^{> e}\ordsum \opord a^{\leq e}$ is the decomposition in  all terms of degree respectively bigger than $e$ and at most $e$.

\subsection{The length of a partial well-order}\label{s:odim} 

Let $P$ be a partial well-order. We define 
    the \emph{\hdim} $\hdmod P \cdot$ on $P$ by transfinite induction as follows: at successor stages, we say that $\hdmod P a\geq\opord r\ssum 1$, if there exists $b< a$ with $\hdmod P b\geq\opord r$, and at limit stages, that $\hdmod Pa\geq\opord r$, if there exists for each $\opord a<\opord r$ some $b_{\opord a}\leq a$ with $\hdmod P{b_{\opord a}}\geq\opord a$. We then say that $\hdmod Pa=\opord r$ if $\hdmod Pa\geq\opord r$ but not $\hdmod Pa\geq \opord r\ssum 1$. In particular, $\hdmod P\bot=0$. For a subset $A\sub P$, we set $\hdmod PA$ equal to the supremum of all $\hdmod Pa$ with $a\in A$. Finally, we define the \emph{(ordinal) length} of $P$ as $\len P:=\hdmod PP=\sup\{\hdmod Pa|a\in P\}$. If $P$ has a maximal element $\top$, then $\len P=\hdmod P\top$.

\begin{example}[Ordinals]\label{E:lenord}
Note of caution: the length of an ordinal $\opord a$ can be different from the ordinal itself. Indeed, if $\opord a$ is a successor ordinal, with predecessor $\opord a'$, then, as a chain, it is given by $0<1<\dots<\opord a'$, with maximal element $\top_{\opord a}=\opord a'$. An easy induction shows that $\hdmod{\opord a}{\opord b}=\opord b$, and so $\len{\opord a}=\hdmod{\opord a}{\top_{\opord a}}=\hdmod{\opord a}{{\opord a}'}={\opord a}'$. On the other hand, if ${\opord a}$ is a limit ordinal, then it has no maximal element, and $\len {\opord a}$ is the supremum of all $\hdmod{\opord a}{\opord b}={\opord b}$ with ${\opord b}<{\opord a}$, that is to say, $\len {\opord a}={\opord a}$. We may summarize this into a single formula
\begin{equation}\label{eq:lenord}
\len{\opord a}=\sup\term{\opord a}0
\end{equation} 
\end{example} 

Let us say that a partial well-order $P$ \emph{admits a composition series}, if there exists a chain $\mathcal C$ in $P$ with $\len P=\len{\mathcal C}$. Not every partial well-order has a composition series as the following example shows:

\begin{example}\label{E:noncomp}
Let $P'$ be the disjoint union of all finite ordinals (meaning that there are no order relations among the different disjuncts) and let $P$ be obtained from $P'$ by adding a single element $\top$ above all elements in $P'$. Hence $\hdmod P{\top}=\omega$, but any chain in $P$ has finite length. It is true that $\len P$ is equal to the supremum of all $\len{\mathcal C}$ with $\mathcal C$ a chain in $P$. However, if instead we let $Q$ be obtained from $P'$ by adding two elements $a$ and $\top$ above each element in $P'$ with $a<\top$, then $\len Q=\hdmod Q\top=\omega\ssum 1$ but the supremum of all chain lengths is just $\omega$, so even the supremum of all chain lengths is less than the actual length.
\end{example} 

\begin{lemma}\label{L:AB}
Let $P$ be a  partial well-order and let 
$A,B\sub P$ be subsets. If $A\leq B$, then 
\begin{equation}
\label{eq:AB}
\hdmod AA\ordsum \hdmod BB \leq \hdmod PB
\end{equation}
\end{lemma}
\begin{proof}
Let ${\opord a}:=\hdmod AA=\len A$. Since  $\hdmod PB$ is the supremum of all $\hdmod Pb$ with $b\in B$, it suffices to show   that
\begin{equation}\label{eq:Ab}
{\opord a}\ordsum \hdmod Bb \leq  \hdmod Pb.
\end{equation}
We will prove \eqref{eq:Ab} by induction on ${\opord b}:=\hdmod Bb$. Assume
first that ${\opord b}=0$. Let ${\opord p}:=\hdmod PA$. Since ${\opord a}$
is the supremum of all $\hdmod Aa$ for $a\in
A$, and since $\hdmod Aa\leq\hdmod Pa$, we get ${\opord a}\leq{\opord p}$. Since $A\leq b$, we have
${\opord p}\leq\hdmod Pb$, and hence we are done in this case.

Next,  assume   ${\opord b}$ is a successor ordinal with predecessor by ${\opord b}'$. By definition, there exists $b'\in B$ below $b$ such that $\hdmod B{b'}\geq{\opord b}'$.
By induction,
we get $\hdmod  P{b'}\geq {\opord a}\ordsum {\opord b}'$. This in turn shows that
$\hdmod Pb$ is at least ${\opord a}\ordsum {\opord b}$. Finally, assume ${\opord b}$ is a limit
ordinal. Hence for each ${\opord c}<{\opord b}$, there
exists
$b_{\opord c}\in B$ below $b$ such that $\hdmod B{b_{\opord c}}={\opord c}$. By induction,
$\hdmod P{b_{\opord c}}\geq{\opord a}\ordsum {\opord c}\leq {\opord a}\ordsum {\opord b}$ and hence also  $\hdmod Pb\geq {\opord a}\ordsum {\opord b}$.
\end{proof}

\subsection{Sum Orders}\label{s:sum}
By the  \emph{sum} $P+Q$ of two  partially ordered sets $P$ and $Q$ (which, after taking an isomorphic copy,  we may assume to be disjoint), we mean the partial order induced on their  union $P\sqcup Q$ by declaring any element in $P$ to lie below any element in $Q$. In fact, if ${\opord a}$ and ${\opord b}$ are ordinals, then their ordinal sum $\opord a\ordsum\opord b$ is just ${\opord a}\sqcup{\opord b}$.
We may represent elements
in the disjoint union $P\sqcup Q$ as pairs $(i,a)$ with $i=0$ if $a\in P$
and $i=1$ if $a\in Q$. The
ordering  $P+ Q$ is  then the lexicographical
ordering on such pairs, that is to say, $(i,a)\leq (j,b)$ if $i<j$ or if $i=j$
and $a\leq b$. 

\begin{proposition}\label{P:sum}
If $P$ and $Q$ are partial well-orders, then so is $P+Q$.
If $P$ has moreover a \maxel, then 
$\len{P+Q}=\len P\ordsum \len Q$.
\end{proposition}
\begin{proof}
We leave it as an exercise to show that $P+ Q$ is a partial well-order.  
Let ${\opord p} :=\hdmod P{\top_P}= \len P$. For a pair $(i,a)$ in  $P+ Q$, let ${\opord n}(i,a)$ be
equal to $\hdmod  Pa$ if $i=0$
and to ${\opord p}\ordsum \hdmod Qa$ if $i=1$. The assertion will follow once we showed that
${\opord n}(i,a)=\hd{i,a}$, for all $(i,a)\in P+Q$, where we wrote $\hd{i,a}$ for $\hdmod {P+Q}{i,a}$. We use transfinite induction. If $i=0$, that is
to say, if $a\in P$, then the claim is easy to check, since no element from $Q$
lies below $a$. So we may assume $i=1$ and $a\in Q$.  Let ${\opord a}:=\hdmod  Qa$
and suppose first that ${\opord a}=0$. Since any element of $P$ lies below $a$, in
any case ${\opord p}\leq\hd{i,a}$. If this were strict, then there would be
an element  $(j,b)$ below $(i,a)$ of \hdim\ ${\opord p}$.
Lest
$\hd{0,\top_P}$ would be bigger than ${\opord p}$, we must have $j=1$ whence $b\in Q$. Since $b\leq_Q a$, we get  $\hdmod Q a\geq1$, contradiction.
  This
concludes the case ${\opord a}=0$, so  assume ${\opord a}>0$. We leave the limit case to
the reader and assume moreover that ${\opord a}$ is a successor ordinal with predecessor ${\opord a}'$. Hence there
exists some  $b\in Q$   below $a$    with $\odmod Q{b}={\opord a}'$. By
induction, $\hd{1,b}={\opord n}(1,b)={\opord p}\ordsum {\opord a}'$, and hence 
$\hd{1,a}\geq{\opord p}\ordsum {\opord a}$. By a similar argument as above, one then easily
shows that this must in fact be an equality, as we wanted to show.
\end{proof}

\subsection*{Increasing functions}
 Let $f\colon P\to Q$ be an increasing (=order-preserving) map
between ordered sets. We say that $f$ is
\emph{strictly
increasing}, if $a<b$ then $f(a)<f(b)$. For instance, an increasing, injective  map  is strictly increasing.
%

\begin{theorem}\label{T:inc}
Let $f\colon P\to Q$ be a strictly increasing map between partial well-orders. If $P$ has  a \minel\  $\bot_P$, then  
\begin{equation*}
\hdmod Q{f(\bot_P)}\ordsum \hdmod Pa\leq \hdmod Q{f(a)}.
\end{equation*}
for all    $a\in P$.
\end{theorem}
\begin{proof}
From the context, it will  be clear in which ordered set we calculate
the rank and hence we will drop the superscripts.   Let ${\opord c}:=\hd{f(\bot)}$.  We
   induct  on ${\opord a}:=\hd a$, where the case
${\opord a}=0$ holds trivially.    We  leave the limit case to the reader and  
assume that ${\opord a}$
is a successor ordinal  with predecessor ${\opord a}'$. By definition, there exists $b<a$
with $\hd{b}={\opord a}'$. By induction, the \hdim\ of
$f(b)$ is at least ${\opord c}\ordsum {\opord a}'$. By assumption,  $f(b)<f(a)$, showing that  $f(a)$ has \hdim\ at least ${\opord c}\ordsum {\opord a}$. 
\end{proof}

Even in the absence of a \minel, the inequality   still holds,  upon replacing the first ordinal in the formula by the minimum of the ranks
of all $f(a)$ for $a\in P$. In particular, \hdim\  
always
increases.

\section{Semi-additivity}\label{s:mod}
 Let $R$ be a ring and $M$ a Noetherian $R$-module.
The \emph\sch\ of $M$ (over $R$) is by definition the collection $\grass  RM$ of all submodules
of $M$, ordered  by reverse inclusion.  
The \hdim\ of $\gr M$ will be called the \emph{length} $\lenmod RM$ of $M$ as
an $R$-module. This is well-defined, since   $\grass RM$ is   a well-partial order.
Thus, for   $N\sub M$, we have $\hd N\geq {\opord a}\ssum 1$, if there exists $N'$ properly  containing $N$ with $\hd{N'}\geq{\opord a}$; and $\len M$ is then given as $\hd{\zeroid M}$, where $\zeroid M$ denotes the
zero submodule of $M$. 
 Since the initial closed interval 
$\init {\gr M}N$   is isomorphic to $\grass R {M/N}$, we get
 \begin{equation}\label{eq:quot}
 \hd N=\hdmod {\grass R M}N=\hdmod{\grass R{M/N}}{\zeroid {M/N}}=\lenmod R {M/N}.
 \end{equation}
Similarly, $\term{\grass R M} N$ consists of all submodules of $M$ contained in $N$, whence is equal to $\grass R N$. Note that if $I$ is an ideal in the annihilator of $M$, then $\grass RM=\grass {R/I}M$, so that in order to calculate the length  of $M$, it makes no difference whether we view it as an $R$-module or as an $R/I$-module. We call the \emph{length} of $R$, denoted $\len R$, its length when viewed as a module over itself. 
 Hence, the length  of $R/I$ as an $R$-module is the same as that of $R/I$ viewed as a ring. We define the \emph{order}, $\order RM$, and \emph{valence}, $\val M$, as the respective order and valence of $\lenmod RM$.

\begin{theorem}[Semi-additivity]\label{T:semadd}
If $\Exactseq NMQ$ is an exact sequence of Noetherian $R$-modules, then 
\begin{equation}\label{eq:lensemadd}
  \lenmod R Q\ordsum \lenmod R N\leq \lenmod R M\leq \lenmod R Q\ssum \lenmod R N
\end{equation}
Moreover, if the sequence is split, then the last inequality is an equality.
\end{theorem}
\begin{proof}
The last assertion   follows from the first, \Thm{T:prod}, and    the fact that then 
$$ \grass R N\times\grass R Q\sub \grass R M.$$
 To prove the lower estimate, let $A$ be the initial closed interval $\init {\grass R M}N$ and let $B$ be the terminal closed interval  $\term {\grass R M}N$. By our discussion above, $A=\grass R{M/N}=\gr Q$, since $M/N\iso Q$,  with   \maxel, viewed in $\grass R Q$, equal to   $\zeroid Q$. By the same discussion,  $B=\grass R N$ with \maxel\ $\zeroid N$. Since $A\leq B$, we may apply Lemma~\ref{L:AB}   to get an inequality 
$$
\hdmod {\grass R Q}{\zeroid Q} \ordsum  \hdmod {\grass R N}{\zeroid N}\leq \hdmod{\grass R M}{\zeroid M},
$$
from which the assertion follows.

To prove the upper bound, let $f\colon \grass RM\to \grass RN\times\grass RQ$ be the map sending a submodule $H\sub M$ to the pair $(H\cap N,\pi(H))$, where $\pi$ denotes the morphism $M\to Q$. It is not hard to see that this is an increasing function. Although it is in general not injective, I claim that $f$ is strictly increasing, so that we can apply  \Thm{T:inc}. Together with \Thm{T:prod}, this gives us the desired inequality. So remains to verify the claim: suppose $H<H'$ but $f(H)=f(H')$. Hence $H'\varsubsetneq H$, but $H\cap N=H'\cap N$ and $\pi(H)=\pi(H')$. Applying the last equality to an element $h\in H\setminus H'$, we get $\pi(h)\in \pi(H')$, whence $h\in H'+N$. Hence, there exists $h'\in H'$ such that $h-h'$ lies in $H\cap N$ whence in $H'\cap N$. This in turn would mean $h\in H'$, contradicting our assumption on $h$.
\end{proof}

\begin{corollary}\label{C:regid}
Let $R$ be a Noetherian ring. If $x$ is an $R$-regular element and $I\sub R$ an
arbitrary ideal, then
 $$
\len{R/xR}\ordsum \len{R/I}\leq \len{R/xI}.
$$
\end{corollary}
\begin{proof}
Apply Theorem~\ref{T:semadd} to the exact sequence
$$
\exactseq {R/I} x{R/xI} {} {R/xR}.
$$
\end{proof}
 
Applying Corollary~\ref{C:regid} to the zero ideal and observing that ${\opord a}\ordsum {\opord b}={\opord b}$ \iff\ $\deg{\opord a}<\deg{\opord b}$, we get:

\begin{corollary}\label{C:hyp}
If $x$ is an $R$-regular element, then the  degree of $\len{R/xR}$ is
 strictly less than the degree of $\len R$.\qed
\end{corollary}

\begin{theorem}\label{T:dim}
Let $M$ be a   finitely generated module over a Noetherian ring $R$. Then the degree of $\lenmod RM$ is equal to the dimension of $M$. In particular,   $R$ is a $d$-dimensional domain \iff\ $\len R=\omega^d$.
\end{theorem}
\begin{proof}
Let ${\opord m}$   be the   length of $M$ and $d$ its dimension. 
 We start with proving the inequality  
\begin{equation}\label{eq:dimlow}
 \omega^d\leq {\opord m}. 
\end{equation}
 We will do this first for $M=R$, by induction on  $d$, where the case $d=1$ is clear, since $R$ does not have finite length. Hence we may assume $d>1$.
Taking the residue modulo a $d$-dimensional prime ideal (which only can
lower   length), we may assume that $R$ is a domain. Let $\pr$ be a $(d-1)$-dimensional prime ideal  and let $x$ be a non-zero element in $\pr$. 
By Corollary~\ref{C:hyp},
 the degree of   $\len{R/xR}$ is at most $\deg{\opord m}-1$.  
By induction,
 $\omega^{d-1}\leq\len{R/\pr} \leq \len{R/xR}$, whence $d-1\leq \deg{\opord m}-1$, proving \eqref{eq:dimlow}.

For $M$ an arbitrary $R$-module, let $\pr$ be a $d$-dimensional associated prime of $M$, so that $R/\pr$ is isomorphic to a submodule of $M$, whence  $\len{R/\pr}\leq \opord m$
  by \Thm{T:semadd}. As we already proved that $\omega^d\leq\len{R/\pr}$, we obtain  \eqref{eq:dimlow}.

Next we show that 
\begin{equation}\label{eq:dimhi}
{\opord m}<\omega^{d+1},
\end{equation}
 again by induction on $d$. 
Assume   first that $M=R$ is a domain. Since   $R/I$ has then dimension at most $d-1$ for any non-zero ideal $I$, we get $\len{R/I}<\omega^d$ by our induction hypothesis. By \eqref{eq:quot}, this means that any non-zero ideal has \hdim\ less than $\omega^d$, and hence $R$ itself has length at most $\omega^d$. Together with \eqref{eq:dimlow}, this already proves one direction in the second assertion.
For the general case, we do a second induction, this time on ${\opord m}$. Let $\pr$  be again a $d$-dimensional associated prime of $M$, and consider an exact sequence $\Exactseq {R/\pr}M{\bar M}$. By \Thm{T:semadd}, we get ${\opord m}\leq \len{\bar M}\ssum\len{R/\pr}$.  By what we just proved, $\len{R/\pr}=\omega^d$, and hence by induction ${\opord m}\leq\len{\bar M}\ssum\omega^d<\omega^{d+1}$. The first assertion is now immediate from \eqref{eq:dimlow} and \eqref{eq:dimhi}.

Conversely, if $R$ has length $\omega^d$, then for any non-zero ideal $I$, the length of $R/I$ is strictly less than $\omega^d$, whence its dimension is strictly less than $d$ by what we just proved. This shows that $R$ must be a domain.
\end{proof}

\section{Length as a cohomological rank}\label{s:cohrk}

In this section, $R$ wil always be a Noetherian ring of finite Krull dimension, and all $R$-modules  will be finitely generated. 
We denote the collection of all associated primes of $M$ (= prime ideals of the form $\ann {}a$ with $a\in M$) by $\op{Ass}(M)$; it is always a finite set. 
We will make frequent use, for a short exact sequence $\Exactseq NMQ$, of the following two inclusions (see, for instance, \cite[Lemma 3.6]{Eis})
 
\begin{align}
 \label{eq:ass1} \op{Ass}(N)&\sub\op{Ass}(M);\\
 \label{eq:ass2} \op{Ass}(M)&\sub\op{Ass}(N)\cup\op{Ass}(Q)
\end{align}

 Let $M$ be a finitely generated $R$-module, and $\id\sub R$ an ideal. Recall that the \emph{$\id$-torsion} or (zero-th) \emph{local cohomology} of $M$ at $\id$, denoted $\Gamma_\id(M)$, 
is given as the intersection of all $\ann {M}{\id^n}$, that is to say, as all elements in   $M$ that are killed by some power of $\id$. This is a left exact functor and its higher derived functors will be denoted $\hlc Mj\id$.  Following common practice, we will identify $\Gamma_\id(M)$ with its zero-th derived functor and henceforth denote it  $\lc M\id$. 
If $(R,\maxim)$ is local, then $\lc M\maxim$ has finite length, denoted $\mul M\maxim$ and  called    the  \emph{local multiplicity} of   $M$. 

Note that $\hlc M0\pr_\pr\iso \hlc{M_\pr}0{\pr R_\pr}$. Since localization is exact, we get
\begin{equation}\label{eq:lcloc}
\hlc Mi\pr_\pr\iso \hlc{M_\pr}i{\pr R_\pr}
\end{equation}  
for all $i$. We will write $\mul M\pr$ for $\mul {M_\pr}{\pr R_\pr}$ and call it the \emph{local multiplicity of $\pr$ on $M$};\footnote{See, for instance, \cite[p.~102]{Eis}; not to be confused with the multiplicity of a module at a primary ideal given in terms of its Hilbert function.} it is non-zero \iff\ $\pr\in\op{Ass}(M)$.

 We now define
the \emph{cohomological rank} of a module $M$ as  
$$
\cohrk M:=\Ssum_{\pr} \mul M\pr  \omega^{\op{dim}(R/\pr)}.
$$
It is instructive to view this from the point of view of Chow cycles. 
Let $\mathcal A(R)$ be the \emph{Chow ring} of $R$, defined as the free Abelian group on $\op{Spec}(R)$. An element $D$ of $\mathcal A(R)$ will be called a \emph{cycle}, and  will be represented as a finite sum  $\sum a_i[\pr_i]$, where $[\pr]$ is the symbol denoting the free generator corresponding to the prime ideal $\pr$. The sum of all $a_i$ is called the degree $\op{deg}(D)$ of the cycle $D$. We define a partial order on $\mathcal A(R)$ by the rule that $D\aleq E$, if $a_i\leq b_i$, for all $i$, where $E=\sum b_i[\pr_i]$. In particular, denoting the zero cycle simply by $0$, we call a cycle $D$ \emph{effective} , if $0\aleq D$, and we let $\mathcal A^+(R)$ be the semi-group of effective cycles. This allows us to define a map  from effective cycles to ordinals by sending the effective cycle $D=\sum_ia_i[\pr_i]$ to the ordinal
$$
\binord D:=\Ssum_i a_i\omega^{\op{dim}(R/\pr_i)}.
$$
Clearly, if $D$ and $E$ are effective, then $\binord{D+E}=\binord D\ssum\binord E$. Moreover, if $D\aleq E$, then $\binord D\aleq\binord E$, so that we get a map $(\mathcal A^+(R),+,\aleq)\to (\ord,\ssum,\aleq)$ of partially ordered semi-groups.

To any $R$-module $M$, we can  assign its \emph{fundamental cycle}, by the rule
\begin{equation}\label{eq:}
\fcyc RM:=\sum_{\pr}\mul M \pr [\pr]
\end{equation} 
 In particular, $\binord {\fcyc{} M} =\cohrk M$.  
 Our main result now links this cohomological invariant to our combinatorial length invariant :
 
 \begin{theorem}\label{T:cohrk}
For any finitely generated module $M$ over a finite-dimensional Noetherian ring $R$, we   have $\len M=\cohrk M$.
\end{theorem}


 Before we give the proof, we derive two lemmas. It is important to notice that the first of these is not true at the level of cycles.

\begin{lemma}\label{L:lcord}
If $M\to Q$ is a proper surjective morphism of $R$-modules, then 
$$
\cohrk Q< \cohrk M.
$$
\end{lemma}
\begin{proof}
Let $N$ be the (non-zero) kernel of $M\to Q$, and let $d$ be its dimension. If $\pr\in\op{Ass}(M)$ but not in the support of $N$, then  $M_\pr\iso Q_\pr$, so that they have the same local cohomology. This holds in particular for any $\pr\in\op{Ass}(M)$ with $\op{dim}(R/\pr)>d$, showing that $\cohrk Q$ and $\cohrk M$ can only start differing at a coefficient of $\omega^i$ for $i\leq d$.  So let $\pr\in\op{Ass}(M)\cap\op{Supp}(N)$ have dimension $d$. In general, local cohomology is only left exact, but by \Lem{L:loccohgen} below, we have in fact an exact sequence \eqref{eq:lc}. 
  Since $\mul N \pr \neq0$, we must therefore have $\mul Q\pr <\mul M\pr $.  It now easily follows that $\cohrk Q< \cohrk  M$.
\end{proof}

\begin{lemma}\label{L:loccohgen}
Given an exact sequence $\Exactseq NMQ$, if $\pr$ is a minimal prime of $N$, then 
\begin{equation}\label{eq:lc}
\Exactseq{ \hlc {N_\pr}0{\pr R_\pr}} { \hlc{M_\pr}0{\pr R_\pr}}{ \hlc{Q_\pr}0{\pr R_\pr}}
\end{equation}
is exact.
\end{lemma}
\begin{proof}
Upon localizing, using \eqref{eq:lcloc}, we may assume $(R,\pr)$ is local and $N$ has finite length. We only need to prove exactness at the final map. 
By assumption,   $N$ is annihilated by some power $\pr^n$. 
Let $a\in M$ be such that its image  $\bar a$ in $Q$ lies in $\hlc Q0\pr$,   that is to say, $\pr^m\bar a=0$ in $Q$, for some $m$. Therefore, $\pr^ma\in N$, whence $\pr^{m+n}a=0$ in $M$, showing that $a\in\hlc M0\pr$.
\end{proof}

\begin{corollary}\label{C:semiaddcyc}
Given an exact sequence $\Exactseq NMQ$, if $M$ has no embedded primes, then we have an equality of cycles
\begin{equation}\label{eq:semiaddcyc}
\fcyc RM+D=\fcyc RN+\fcyc RQ
\end{equation} 
where $D$ is an effective cycle supported on $\op{Ass}(Q)\setminus\op{Ass}(M)$. 
\end{corollary} 
\begin{proof}
Let $D$ be the cycle given by \eqref{eq:semiaddcyc}, so that $D$ has support in $\op{Ass}(M)\cup\op{Ass}(Q)$ by \eqref{eq:ass2}. We need to show that $D$ is effective and supported on $\op{Ass}(Q)\setminus\op{Ass}(M)$.
Since any associated prime $\pr$ of $M$ is minimal, $D$ is not supported in $\pr$ by \Lem{L:loccohgen}. On the other hand, any associated prime of $Q$ not in $\op{Ass}(M)$ appears with a positive coefficient in $D$, showing that the latter is   effective.
\end{proof}

\subsection*{Proof of \Thm{T:cohrk}}
Let us first  prove $ \len M\leq  \cohrk M$  by transfinite induction on $\cohrk M$, where the case $\cohrk M=0$ corresponds to $M=0$. Let $N$ be any non-zero submodule of $M$. By Lemma~\ref{L:lcord}, we have $\cohrk {M/N}<\cohrk M$, and hence our induction hypothesis applied to $M/N$ yields $\len{M/N}\leq \cohrk {M/N}<\cohrk M$. Since $\len{M/N}=\hd N$ by \eqref{eq:quot}, continuity yields that $\len M=\hd {\zeroid M}$ can be at most $\cohrk M$, as we needed to show. 

To prove the converse inequality, we induct on the length of $M$. Choose an associated prime $\pr$ of $M$ of minimal dimension, say, $\op{dim}(R/\pr)=e$. By assumption, there exists $m\in M$ such that $\ann Rm=\pr$. Let $H$ be the submodule of $M$ generated by $m$.  Since $H\iso R/\pr$, we get $\cohrk H=\omega^e$, and so by what we already proved, $\len{R/\pr}\leq \omega^e$. By \Thm{T:dim}, this then is an equality.  So we may assume that $Q:=M/H$ is non-zero.  By \Lem{L:loccohgen}, we get $\mul M\pr =\mul Q\pr-1$. 
By semi-additivity, we have an inequality
$$
\len {Q}\ordsum \len H\leq \len M
$$
and therefore, by induction
\begin{equation}\label{eq:IHcd}
\cohrk{Q}\ordsum \omega^e\leq \len M
\end{equation}
Let $\primary$ be any associated prime of $M$ different from  $\pr$. By minimality of dimension, $\primary$  cannot contain $\pr$. In particular, $M_\primary\iso Q_\primary$, whence $\mul M \primary  =\mul  Q\primary $. 
Let $\opord b:=\cohrk Q^{\geq e}$ (see \eqref{eq:splitord}). 
%
Putting together what we proved so far, we can find an ordinal ${\opord a}$ with $\low{\opord a}\geq e$ (stemming from primes associated to $Q$ but not to $M$), such that $\opord b\ssum \omega^e=\cohrk M\ssum{\opord a}$. Since   $\cohrk Q\ordsum \omega^e=\opord b\ssum \omega^e$, we get, from \eqref{eq:IHcd} and the first part,  inequalities
$$
\cohrk M\ssum {\opord a}\leq \len M\leq \cohrk M
$$
which forces ${\opord a}=0$ and  all inequalities to be equalities.\qed

\begin{corollary}\label{C:order}
The order of a  module is the smallest dimension of an associated prime, and its valence   is  the   degree of its fundamental cycle. \qed
\end{corollary} 

 By \cite[Proposition 1.2.13]{BH}, over a local ring, we have 
\begin{equation}\label{eq:depthorder}
\op{depth}(M)\leq \order {}M.
\end{equation} 
This inequality can be strict: for example, a two-dimensional domain which is not \CM, has depth one but order two by \Thm{T:dim}.  As an illustration of the use of \Thm{T:cohrk}, let us calculate the length of some special modules.  

\begin{proposition}\label{P:lenproj}
Suppose $P$ is a finitely generated projective module over a $d$-dimensional Noetherian ring $R$. If $P$ has rank $n$, then $\len P=n\cdot\len R$.
\end{proposition}
 \begin{proof}
For any prime ideal $\pr$ of $R$, we have $P_\pr\iso R^n_\pr$, and hence $\mul P\pr=n\mul R\pr$. The result now follows from \Thm{T:cohrk} and \eqref{eq:scalord}.
\end{proof}

Note that this is false if $P$ is a projective module without rank, like the example $P=\zet/2\zet$ over the ring $R=\zet/6\zet$ already shows. By \Thm{T:cohrk}, any $d$-dimensional   local \CM\ ring $R$ has length $e\omega^d$, where $e=\mul Rd$ is its generic length (see \Cor{C:topadd} below), since any associated prime is $d$-dimensional (see, for instance, \cite[Theorem 17.4]{Mats}). In a future paper, we will show that $e$ is at most the multiplicity of $R$.

\begin{proposition}\label{P:lencan}
If a \CM\ local ring $R$ admits a canonical module $\omega_R$, then $\len R=\lenmod R{\omega_R}$.
\end{proposition} 
\begin{proof}
Since $\omega_R$ is a maximal \CM\ module, it has the same associated primes $\pr$ as $R$. Given such a $\pr$, since $R_\pr$ is Artinian, and since canonical modules localize, $(\omega_R)_\pr$ is the injective hull of the residue field of $R_\pr$. By \cite[Proposition 3.2.12(e)]{BH}, the length of this injective hull is equal to that of $R_\pr$, showing that $\lcl\pr (R)=\lcl\pr(\omega_R)$, for every prime ideal $\pr$ of $R$. The result now follows from \Thm{T:cohrk}.
\end{proof}

  Our next result gives a constraint on the possible length  of a submodule, which is exploited in \cite{SchBinEndo}  to study binary modules. Let us say that $\opord a$ is \emph{weaker than} $\opord b$, denoted  $\opord a\aleq\opord b$,  if $\mul{\opord a}i\leq\mul{\opord b}i$, for all $i$. Clearly, $\opord a\aleq \opord b$ implies $\opord a\leq\opord b$, but the converse fails in general (e.g., $\omega$ is smaller than $\omega^2$ but not weaker than it).

\begin{theorem}\label{T:submod}
If $N\sub M$, then $\len N\preceq\len M$. Conversely, if ${\opord n}\aleq\len M$, then there exists a submodule $N\sub M$ of length ${\opord n}$. 
\end{theorem}
\begin{proof}
The first assertion is immediate from \Thm{T:cohrk}, inclusion \eqref{eq:ass1}, and the fact that \eqref{eq:lc} is always left exact.
For the second assertion, let ${\opord m}:=\len M$. 
 We induct on the (finite collection)  of ordinals ${\opord n}$ weaker than ${\opord m}$ to show that there exists a submodule of that length. The case ${\opord n}=0$ being trivial, we may assume ${\opord n}\neq 0$.
Let $i$ be the order of ${\opord n}$ and write ${\opord n}={\opord q}\ssum\omega^i$ for some ${\opord q}\aleq{\opord n}$. Since then ${\opord q}\aleq{\opord m}$, there exists a submodule of length ${\opord q}$ by induction. Let  $H\sub M$ be   maximal among all submodules of   length ${\opord q}$. By \Thm{T:cohrk}, there exists an $i$-dimensional associated prime $\pr$ of   $M$, such that $\hlc{H_\pr}0{\pr R_\pr}$ is strictly contained in $\hlc{M_\pr}0{\pr R_\pr}$. Hence we can find $x\in M$ outside $H$ such that $\pr x\sub H$. Let $N:=H+Rx$ and let $\bar x$ be the image of $x$ modulo $H$, so that  $N/H=R\bar x$. Since $\pr\bar x=0$, the length of $R\bar x$ is at most $\omega^i$. By semi-additivity applied to the inclusion $H\sub N$, we have an inequality $\len N\leq {\opord q}\ssum\ \len{R\bar x}$, and hence $\len N\leq {\opord n}$. Maximality of $H$ yields  ${\opord q}<\len N$. On the other hand, since   $\len N\preceq{\opord m}$ by our first assertion, minimality of $i$ then  forces $\len N={\opord n}$, as we needed to show. 
\end{proof}  

\begin{remark}\label{R:submod}
In fact, if $N\sub M$, then $\fcyc {}N\aleq\fcyc {}M$, so that the   fundamental cycle map  is a morphism $\gr M^\circ\to \mathcal A(R)$ of partially ordered sets, where $\gr M^\circ$ is the opposite order given by inclusion.  On the other hand, by \eqref{eq:quot} and \Thm{T:cohrk}, the map $\gr M\to \mathcal A(R)$ given by $N\mapsto \fcyc{}{M/N}$ factors through the length map $\gr M\to \ord$, but there is no natural ordering on $\mathcal A(R)$ for which this becomes a map of ordered sets.
\end{remark} 

We may improve the lower semi-additivity by replacing $\leq $ by $\aleq$:

\begin{corollary}\label{C:semadd}
If $\Exactseq NMQ$ is exact, then   $  \len Q\ordsum \len N\aleq \len M$.
\end{corollary} 
\begin{proof}
This is really just a fact about ordinals. Let  ${\opord m},{\opord n},{\opord q}$ be the respective lengths of $M$, $N$, $Q$. By semi-additivity, ${\opord q}\ordsum {\opord n}\leq{\opord m}\leq{\opord q}\ssum{\opord n}$, whereas \Thm{T:submod} gives ${\opord n}\aleq{\opord m}$, and we now show that these inequalities imply that ${\opord q}\ordsum {\opord n}\aleq{\opord m}$.  Write ${\opord m}={\opord n}\ssum{\opord a}$ and let $d$ be the dimension of ${\opord n}$. 
For an arbitrary ordinal ${\opord b}$,  we have a unique decomposition ${\opord b}={\opord b}^{\geq d}\ordsum{\opord b}^{<d}$ as described in \eqref{eq:splitord}. 
By assumption, ${\opord n}^{\geq d}=a\omega^d$ with $a=\mul{\opord n}d$, and hence the semi-additivity inequalities at  degree $d$ and higher become ${\opord q}^{\geq d}\ssum  a\omega^d\leq {\opord a}^{\geq d}\ssum a\omega^d\leq{\opord q}^{\geq d}\ssum a\omega^d$, showing that ${\opord q}^{\geq d}={\opord a}^{\geq d}$. 
By definition of ordinal sum, ${\opord q}\ordsum {\opord n}={\opord q}^{\geq d}\ssum{\opord n}$ and so 
$$
({\opord q}\ordsum {\opord n})\ssum{\opord a}^{< d}={\opord q}^{\geq d}\ssum{\opord n}\ssum{\opord a}^{<d}={\opord a}^{\geq d}\ssum{\opord a}^{<d}\ssum{\opord n}={\opord a}\ssum{\opord n}={\opord m}
$$
proving that ${\opord q}\ordsum {\opord n}\aleq{\opord m}$. 
\end{proof} 

To give a more detailed version of semi-additivity, let us write $\mul Mi:=\mul  {\len M}i$ for the $i$-th Cantor coefficient of $\len M$. By \Thm{T:cohrk}, each $\mul Mi$ is equal to the sum of all $\mul M\pr  $, where $\pr$ runs over all $i$-dimensional associated primes of $M$.


\begin{proposition}\label{P:topadd}
Let $\Exactseq NMQ$ be an exact sequence of modules and let $r,d,s$ be the respective dimensions of $N$, $M$, and $Q$. 
\begin{enumerate}
\item\label{i:r} $\mul Mr=\mul Nr+\mul Qr$, 
\item\label{i:rplus} If $r<d$, then $d=s$ and $\mul Mi=\mul Qi$, for all $r<i\leq d$,
\item\label{i:qplus} If $s<d$, then $r=d$ and $\mul Ni=\mul Mi$, for all $s<i\leq d$.
\end{enumerate}
\end{proposition} 
\begin{proof}
Write the length of $N$, $M$, and $Q$, respectively as $\opord n:=\Splus_{i=0}^r n_i\omega^i$, $\opord m:=\Splus_{i=0}^d m_i\omega^i$, and  $\opord q:=\Splus_{i=0}^s q_i\omega^i$, with $n_r,m_d,q_s\neq0$. In particular
$$
\opord q\ordsum\opord n=\Splus_{i=r+1}^d q_i\omega^i\splus (q_r+n_r)\omega^r\splus\Splus_{i<r}n_i\omega^i.
$$
By semi-additivity, we have 
\begin{equation}\label{eq:semaddqnm}
\opord q\ordsum\opord n\leq\opord m\leq\opord q\splus\opord n, 
\end{equation} 
and the first two relations follow now easily by comparing Cantor coefficients. 

So assume $s<d$, so that $\opord q\ordsum\opord n=\opord n$. We prove by downward induction on $t>s$ that $n_t=m_t$. The case $t=d$ is covered by \eqref{i:r}. So assume we have already proven this for all $i>t$. In particular,  each of the  three ordinals in \eqref{eq:semaddqnm} have the same  part of degree $t+1$ and higher, and so we may subtract it from each of them. The resulting inequality has become
$$
n_t\omega^t+\opord n^{<t}\leq m_t\omega^t+\opord m^{<t}\leq n_t\omega^t+\opord v
$$
with $\opord v$ of degree at most $t-1$ (note that by assumption $\op{deg}(\opord q)\leq t-1$). Hence, the leading coefficients must  be equal, that is to say, $n_t=m_t$, and so we are done by induction.
\end{proof}


Since $\dim N\leq \dim R$, we immediately get:

\begin{corollary}[Top additivity]\label{C:topadd}
If $R$ has dimension $d$, then $\mul\cdot d$ (called   the \emph{generic length}) is additive on exact sequences.\qed
\end{corollary} 

I conclude this section with an example illustrating how length captures the nilpotent structure of a ring (the proof will be given in a future paper on the length of monomial algebras).

\begin{example}\label{E:len}
Let $R$ be the quotient of the polynomial ring $\pol k{x,y,z}$ modulo the monomial ideal
$$
I:=(x^5yz,x^4y^3z,x^3y^2z^2,x^2y^3z^2, x^2y^2z^3, x^4yz^4,x^6z^4,x^5z^5).
$$
Without proof we state that the minimal primes are $(x)$ and $(z)$ with respective local multiplicities $2$ and $1$; the one-dimensional primes are $(x,y)$, $(x,z)$, $(y,z)$, with respective multiplicities $5$, $1$, and $3$;  and the maximal ideal $(x,y,z)$ has local multiplicity $7$. Therefore, by \Thm{T:cohrk}, the  length is  
$$
\len R=3\omega^2+9\omega+7.
$$
\end{example} 

\section{Base change}
The behavior of length under base change is intricate, and so we will only discuss some basic facts. Recall that a \homo\ $R\to S$ is called \emph{cyclically pure}, if $I=IS\cap R$ for every ideal $I\sub R$. Faithfully flat maps are examples of cyclically pure \homo{s}.

\begin{lemma}\label{L:ff}
If $R\to S$ is   cyclically pure, then $\len R\leq\len S$.
\end{lemma}
\begin{proof}
Consider the  canonical map $f\colon  \gr R \to \gr S$ given by sending an ideal $I\sub R$ to its extension $IS$.    Being cyclically pure means that  $f$ is injective, and the inequality is
 now immediate from Theorem~\ref{T:inc}.
\end{proof}

Even under extensions of scalars can this be a strict inequality: let $R$ be the domain $\pol{\mathbb R}{x,y}/(x^2+y^2)$ and note that $R\tensor_{\mathbb R}\mathbb C$ is not a domain and hence must have  length different from $\len R=\omega$  by \Thm{T:dim}. Using \Thm{T:cohrk}, one easily calculates that $\len{R\tensor_{\mathbb R}\mathbb C}=2\omega$. A similar phenomenon occurs for completion: let $R$ be the analytically reduced domain $\pol k{x,y}/(x^2+y^2+y^3)$, so that its completion $\complet R$ is not a domain and hence has length different from $\len R=\omega$ (again using \Thm{T:cohrk}, one calculates that $\len{\complet R}=2\omega$). Another factor to take into account, is that the rings might have different dimension.  Apart from the latter, length behaves well under polynomial extensions:

\begin{corollary}\label{C:polext}
Let $R$ be a finite-dimensional Noetherian ring and $t$ an $n$-tuple of variables. Then $\len {\pol Rt}=\omega^n\cdot \len R $. 
\end{corollary}
\begin{proof}
By an inductive argument, it suffices to treat the case that $t$ is a single variable. The associated primes of $\pol Rt$ are precisely the extensions $\pr\pol Rt$ with $\pr\in\op{Ass}(R)$. Moreover, $\lcl {\pol Rt}(\pr\pol Rt)=\lcl R(\pr)$. Since $\pol Rt/\pr\pol Rt\iso \pol{(R/\pr)}t$,   its dimension is equal to $\dim {(R/\pr)}+1$, and the result now follows from \Thm{T:cohrk}.
\end{proof}

We can now identify one useful class of extensions that preserve length: recall that a \emph{Nagata extension} of a local ring $(R,\maxim)$ is a localization of some polynomial extension $\pol Rt$  with respect to the prime ideal $\maxim\pol Rt$, and will be denoted $R(t)$, where $t$ is some tuple of variables. The extension $R\to R(t)$ is a scalar extension in the terminology of \cite{SchUlBook}, that is to say, faithfully flat and unramified. In particular, both rings have the same dimension, and so an argument similar but easier as the above gives:

\begin{corollary}\label{C:Nagext}
Nagata extensions do not change the length, that is to say, $\len R=\len{R(t)}$, for any Noetherian local ring $R$ and any tuple of variables $t$.\qed
\end{corollary}

\section{Equilateral sequences}
Let us call an exact sequence $\Exactseq NMQ$   \emph\add,  if the length of $M$ is equal to  $\len N\ssum\len Q$ (the upperbound in \eqref{eq:lensemadd}); if this is moreover also equal to    $\len Q\ordsum \len N$ (the lowerbound in \eqref{eq:lensemadd}), then we call it   \emph{strongly \add}. Similarly, we say that $\Exactseq NMQ$ is \emph{\add\ up to degree $r$}, if $\len N^{\geq r}\ssum\len Q^{\geq r}=\len M^{\geq r}$, that is to say, if $\mul Ni+\mul Qi=\mul Mi$, for all $i\geq r$. Immediately from \Prop{P:topadd}, we get:

\begin{corollary}\label{C:equilattop}
 An exact sequence $\Exactseq NMQ$ is \add\ up to degree $r$, where $r$ is the minimum of $\dim N$ and $\dim Q+1$.\qed
\end{corollary} 

By \Thm{T:semadd}, split exact sequences are \add. More generally, suppose we can embed    $N\oplus Q$ into $M$, then, if $\Exactseq NMQ$ is exact,  it is \add: semi-additivity gives $\len M\leq \len N\ssum\len Q$ whereas $N\oplus Q\into M$ gives $\len Q\ssum\len N\leq\len M$.

\begin{proposition}\label{P:equiorddim}
An exact sequence $\Exactseq NMQ$  such that $\dim Q<\order{}N$ is \add.
\end{proposition} 
\begin{proof}
Let $\opord n$, $\opord m$, and $\opord q$ be the lengths of $N$, $M$, and $Q$ respectively. By \Thm{T:submod}, we have $\opord n\aleq{}\opord m$ and $\opord q\aleq{}\opord m$, whence  $\opord n\of\opord q\aleq{}\opord m$, where $\opord n\of\opord q$ denotes the  $\aleq$-supremum of $\opord n$ and $\opord q$. However, since the degree of $\opord q$ is less than the order of $\opord n$, we have $\opord n\of\opord q=\opord n\ssum\opord q$. Semi-additivity, on the other hand, yields the converse inequality $\opord m\leq \opord n\ssum\opord q$.
\end{proof} 

Note that such a sequence will never be strongly \add. In fact, being strongly \add\ is really a property of ordinals:   ${\opord a}\ordsum {\opord b}={\opord a}\ssum{\opord b}$ \iff\ the degree of ${\opord b}$ is at most the order of ${\opord a}$, and hence

\begin{corollary}\label{C:orddim}
An exact sequence of finitely generated $R$-modules 
$$\Exactseq NMQ,$$
   is strongly \add\ \iff\  $\dim N\leq \order {}Q$.\qed
\end{corollary} 

%
%

Immediately from semi-additivity,  the left exactness of local cohomology and \Thm{T:cohrk}, we get:

\begin{proposition}\label{P:cohtriv}
A short exact sequence $\exactseq N{}MfQ$ is \add\ \iff\ $f$ is a \emph{cohomological epimorphism}, meaning that $f_\pr\colon \hlc M0\pr\to \hlc Q0\pr$ is surjective, for all prime ideals $\pr$.\qed
\end{proposition}

In this paper,\footnote{Our terminology differs slightly from the literature, where, at least for local rings, the condition is formulated  over the completion.} we will call a module $M$  \emph{unmixed}  of dimension $d$ if all its associated primes have dimension $d$.  \Thm{T:cohrk} then shows that   $M$ is unmixed 
 \iff\ $\order {}M=\dim M$,    in which case $\len M=\mul Md\omega^d$. \Cor{C:orddim} yields:

\begin{corollary}\label{C:unmadd}
Let $R$ be a $d$-dimensional Noetherian ring and   $\Exactseq NMQ$  a short exact sequence.   If $Q$ is unmixed of dimension $d$, then this sequence is strongly \add.\qed
\end{corollary} 

Recall the \emph{dimension filtration} $\fl d_0(M)\sub \fl d_1(M)\sub\dots\sub \fl d_d(M)=M$ of a $d$-dimensional finitely generated $R$-module $M$ defined by Schenzel in \cite{ScheDimFil}, where $\fl d_i(M)$ is the submodule of  all elements of dimension at most $i$, and  where we define the \emph{dimension} of an element $x\in M$ as the dimension of the module it generates, that is to say,   $\op{dim}(R/\ann Rx)$. Equivalently, $\fl d_i(M)$ is the largest   submodule of $M$ of dimension at most $i$.

\begin{proposition}\label{P:dimfil}
Given a  $d$-dimensional  module $M$, the exact sequence 
$$\Exactseq{\fl d_i(M)}M{M/\fl d_i(M)}$$
is  {strongly \add}, for each $i$. In particular, $\len{\fl d_i(M)}=\len M^{\leq i}$, for all $i$.
\end{proposition} 
\begin{proof}
It follows from  \cite[Corollary 2.3]{ScheDimFil} that the associated primes of $M/\fl d_i(M)$ are precisely the associated primes of $M$ of dimension strictly larger than $i$. By \Cor{C:order}, this means that $\len{M/\fl d_i(M)}$ has order at least $i+1$. Since $\len{\fl d_i(M)}$ has degree at most $i$ by \Thm{T:dim}, the result follows from \Cor{C:orddim}.
%
\end{proof} 

Another way to formulate this result is  as the following formula  for calculating length 
\begin{equation}\label{eq:dimfil}
\len M=\Ssum_{i=0}^d \len{\fl d_i(M)/\fl d_{i-1}(M)},
\end{equation} 
and each non-zero $\fl d_i(M)/\fl d_{i-1}(M)$ is unmixed of dimension $i$ and of length $a_i\omega^i$, where $a_i$ is its generic length.

%

\begin{example}\label{E:nulen}
Let $R$ be the coordinate ring of a plane with an embedded line inside three dimensional space over $k$ given by the equations $x^2=xy=0$ in the three variables $x,y,z$. Using \Thm{T:cohrk}, one easily calculates that $\len R=\omega^2\ssum \omega$, where the associated primes are $\pr=(x)$ and $\mathfrak q=(x,y)$. The ideals of length $\omega$ are exactly those contained in $\pr$. 
The ideals of length $\omega^2\ssum \omega$ (the \emph{open}  ideals in the terminology from   \S\ref{s:open}), are precisely those 
that contain a non-zero multiple of $x$ and a non-zero multiple of $y$ (this follows, for instance,  from \cite[Proposition 3.10]{SchBinEndo}). Finally, the remaining (non-zero) ideals of length $\omega^2$, are those contained in $\mathfrak q$ but disjoint from $\pr$ (note that if $I$ is not contained in $\mathfrak q$, then $IR_\pr=R_\pr$ and $IR_{\mathfrak q}=R_{\mathfrak q}$, so that $I$ must be open). 
\end{example} 
 
In a future paper, we will use \add\ sequences to define a (new) Grothendieck group on the category of finitely generated $R$-modules, namely, let  $\addgr R$ be the free Abelian group on isomorphism classes $[M]$ modulo the  relations $[M]-[N]-[Q]$, for every \add\ exact sequence $\Exactseq NMQ$. Sending a class of a module $M$ in $\addgr R$ to its length, is a surjective semi-group homomorphism, when we view the ordinals as a semi-group with respect to the shuffle sum $\ssum$, and hence in particular, $\addgr R$ is non-trivial.

\section{Compositions series}\label{s:compser}

We will call a strictly descending chain $\mathcal C$  of submodules of $M$ simply a \emph{chain} in $M$. When discussing composition series, we may assume that a chain has   a first and last element, so that we can represent it as  $\mathcal C:=(M_{\opord a}|{\opord a}\leq\opord r)$.   Since $\mathcal C\iso \opord r\ssum 1$, formula~\eqref{eq:lenord} then yields $\len{\mathcal C}=\opord r$.   
The factor module $M_{\opord a}/M_{{\opord a}\ssum 1}$ will be called the ${\opord a}$-th \emph{cokernel} of $\mathcal C$. Our goal is to show that each module $M$ admits a composition series $\mathcal C$, that is to say, a chain $\mathcal C$ such  that $\len M=\len{\mathcal C}$. In other words,  a partial well-order as in \Examp{E:noncomp} can never be a \sch.

\begin{theorem}\label{T:chain}
Every finitely generated module over a Noetherian ring of finite Krull dimension admits a composition series.
\end{theorem} 
\begin{proof}
We induct on the 
length ${\opord m}:=\len M$. For ${\opord m}<\omega$, this is just the classical Jordan-Holder theorem for modules of finite length (note that the length of a finite chain is defined to be one less than the number of members in the chain, explaining why we have to take the length of the ordinal, not the ordinal itself). So assume $\opord m\geq\omega$. If ${\opord m}$ has valence at least two, we can write it as $\bar{\opord m}\ssum\omega^e$, where $e$ is the order of ${\opord m}$. Let $N$ be a submodule of \hdim\ $\bar{\opord m}$, so that $\bar M:=M/N$ has length $\bar{\opord m}$ by \eqref{eq:quot}. If ${\opord n}:=\len N$, then by semi-additivity, we get $\bar{\opord m}\ordsum {\opord n}\leq \bar{\opord m}\ssum\omega^e\leq\bar{\opord m}\ssum{\opord n}$. The latter implies that $\omega^e\leq{\opord n}$, and by the former inequality, it cannot be bigger either. By induction we can find a chain $\bar{\mathcal C}$ in $\bar M$ with $\len{\bar{\mathcal C}}=\len{\bar M}$. Taking the pre-image of each module in $\bar{\mathcal C}$ in $M$ gives a chain $ {\mathcal C}$ in $M$, of length $\bar{\opord m}$, in which each term contains $N$. %
 Induction also gives a chain $\mathcal D$ in $N$ with $\len{\mathcal D}=\len N=\omega^e$. The union $\mathcal C\cup\mathcal D$ is therefore equal to the chain $\mathcal C+\mathcal D$, which by \Prop{P:sum} has length  $\len{\mathcal C}\ordsum \len{\mathcal D}=\bar{\opord m}\ssum \omega^e={\opord m}$, showing that it is a composition series in $M$.

So remains the case that ${\opord m}=\omega^d$.  We induct this time on $d$, where the case $d=1$ is classical: if there were no infinite chains,  then $M$ is both Artinian and Noetherian, whence of finite length (\cite[Proposition 6.8]{AtiMac}). Put $M_0:=M$ and choose a submodule $M_1$ of \hdim\ $\omega^{d-1}$. Applying the induction hypothesis to  the quotient $M/M_1$, which has  length $\omega^{d-1}$, we can find, as above, a chain ${\mathcal C}_1$ of submodules   containing $M_1$, of length $\omega^{d-1}$. By semi-additivity, the length of $M_1$ is again $\omega^d$. Choose a submodule $M_2$ of $M_1$ of \hdim\ $\omega^{d-1}$ in $\grass{}{M_1}$, and as before, find a $\omega^{d-1}$-chain ${\mathcal C}_2$ in this \sch\    of submodules   containing $M_2$. Continuing in this manner, we get an $\omega$-chain $M_0\varsupsetneq M_1\varsupsetneq M_2\varsupsetneq\dots$ and $\omega^{d-1}$-chains ${\mathcal C}_n$ from $M_n$ down to $M_{n+1}$. The union of all these chains is therefore a chain of length $\omega^d$, as we needed to construct.
\end{proof} 

\begin{remark}\label{R:strongcomp}
We could ask for additional properties of a composition series (which would hold automatically for finite ones), such as being maximal, etc. By Zorn's lemma, any composition series can be refined to a maximal one, but, as \Examp{E:maxnoncompser} below shows, being maximal is not a sufficient condition for being a composition series. A more subtle question is the existence of a composition series $(M_{\opord a}|{\opord a}\leq\len M)$ with the property that $\hd{M_{\opord a}}={\opord a}$ for all ${\opord a}$. In a future paper, I will show that they exist for excellent, reduced Jacobson rings and for monomial algebras. 
\end{remark}

\begin{example}\label{E:maxnoncompser}
Let $R:=\pol k{x,y,z}/z(x,y)$. Its associated primes are $\pr=(z)$ and $\mathfrak q=(x,y)$, and both localizations are fields, so that $\len R=\omega^2\ssum \omega$ by \Thm{T:cohrk}. Consider the $\omega$-chain $\id_i:=(x,y,z^i)$ with intersection equal to $\id_\omega=(x,y)$, and continue now with a (continuous) $\omega^2$-chain $\id_{(i+1)\omega\ssum j}:=(x^{j+1}y^i,y^{i+1})$, with intersection equal to $\id_{\omega^2}=(0)$. Hence the total chain has length $\omega^2$ and is maximal, but it is too short to be a composition series.
\end{example}

\section{Open submodules}\label{s:open}
By semi-additivity, $\len N$ is at most $\len M$, and, in fact, $\len N\aleq\len M$ by \Thm{T:submod}.  If $M$ has finite length and $N$ is a proper submodule, then obviously its length must be strictly less, but in the non-Artinian case, nothing excludes this from being an equality. So, we call   $N\sub M$ \emph{open}, if $\len N=\len M$. Immediately   from \Thm{T:cohrk} and \Thm{T:submod}, we get

\begin{corollary}\label{C:openval}
A submodule  $N\sub M$ is open \iff\ $\hlc{N_\pr}0{\pr R_\pr}=\hlc{M_\pr}0{\pr R_\pr}$ for all $\pr$, \iff\ $\fcyc RN=\fcyc RM$, \iff\ $\val N=\val M$. \qed
\end{corollary} 

\begin{remark}\label{R:loccoh}
In particular,   the long exact sequence of local cohomology yields that $N\sub M$ is open \iff\  the canonical morphism $\hlc{M_\pr/N_\pr}0{\pr R_\pr}\to \hlc{N_\pr}1{\pr R_\pr}$ is injective, for every (associated) prime $\pr$ (of $M$).
\end{remark}

\begin{proposition}\label{P:big}
Given an exact sequence $$\Exactseq NMQ,$$  if $\dim(Q)<\order  {}M$, then $N$ is open.   In particular, any non-zero ideal in a domain is open, and more generally, any ideal in an  unmixed ring containing a parameter is open. 
 \end{proposition} 
\begin{proof}
Let ${\opord n}=\sum a_i\omega^i$, ${\opord m}=\sum b_i\omega^i$, and ${\opord q}=\sum c_i\omega^i$ be the respective lengths of $N$, $M$ and $Q$. By semi-additivity, ${\opord m}\leq{\opord q}\ssum{\opord n}$, whence $b_i\leq a_i+c_i$, for all $i$. By assumption, $c_i=0$ whenever $b_i\neq 0$, so that in fact $b_i\leq a_i$, for all $i$, that is to say, ${\opord m}\aleq {\opord n}$. Since the other inequality always holds,  $N$ is open.   To prove the last assertion, let $x$ be a parameter in a $d$-dimensional unmixed ring $R$, so that $R/xR$ has dimension at most  $d-1$. Since $\order{}R=d$ by \Thm{T:cohrk}, our first assertion shows that the ideal $(x)$, and hence any ideal containing $x$, is open.  
\end{proof} 

Let us call a submodule $N\sub M$ \emph{\add}, if $\Exactseq NM{M/N}$ is \add, that is to say, by \eqref{eq:quot}, if $\len M=\len N\ssum\hd N$. Hence a direct summand is \add\ by semi-additivity. By \Cor{C:orddim}, any submodule   $N$ such that  $\dim N\leq \order {}{M/N}$, is \add, but the converse need not hold. 
 
\begin{proposition}\label{P:osplitopen}
A maximal (proper) submodule is either \add\ or open. In particular, if $M$ has positive order, then any maximal submodule is open.
\end{proposition} 
\begin{proof}
Let $N\varsubsetneq M$ be maximal, so that $Q:=M/N$ is simple, of length one. Let ${\opord n}$ and ${\opord m}$ be the respective lengths of $N$ and $M$. By semi-additivity, we have ${\opord n}\leq{\opord m}\leq{\opord n}\ssum 1$. If the former inequality holds, the submodule is open, and if the latter holds, it is \add. The latter case is excluded when $\opord m$ is a limit ordinal, that is to say, when $M$ has positive order.
\end{proof} 

In the ring case, we can even prove:

\begin{proposition}\label{P:lenmax}
If $(R,\maxim)$ is a non-Artinian  local ring, then $\maxim$ is open.
\end{proposition} 
\begin{proof}
  Let ${\opord m}$ and ${\opord r}$ be the respective lengths of $\maxim$ and $R$. We induct on $\opord r$. In view of \Prop{P:osplitopen},  to rule out  that $\maxim$ is \add, we may assume that it is an associated prime. 
%
%
  Choose $x\in R$ with $\ann{}x=\maxim$ and put $\bar R:=R/xR$. Since $xR$ has length one,  \Cor{C:orddim} applied to the exact sequences  $\Exactseq{xR}R{\bar R}$ and $\Exactseq{xR}\maxim{\maxim\bar R}$ yields ${\opord r}=\len{\bar R}\ssum 1$ and ${\opord m}=\len{\maxim\bar R}\ssum 1$. By induction, $\len{\maxim\bar R}=\len{\bar R}$, and hence ${\opord m}={\opord r}$. 
\end{proof} 

\begin{remark}\label{R:lenprim}
As for primary ideals $\mathfrak n$, they will not be open in general if $R$ has depth zero. More precisely, suppose  $\len R=\opord r\ssum n$ with $\opord r$ a limit ordinal and $n\in\nat$. If $\mathfrak n\lc R\maxim  =0$ (which will be the case if $\len{R/\mathfrak n}\geq n$), then $\len {\mathfrak n}=\opord r$. Indeed, the case $n=0$ is trivial, and we may always reduce to this since $\mathfrak n$ is   a module over $R/\lc \maxim R$, and the latter has length $\opord r$ by \Prop{P:big}. 
\end{remark}

\begin{corollary}\label{C:openemb}
If $N\sub M$ is open, then  $\op{Supp}(M/N)$ is nowhere dense in $\op{Supp}(M)$.  The converse   holds if $M$ has no embedded primes. 
\end{corollary} 
\begin{proof}
Let $Q:=M/N$. The condition on the supports means that no minimal prime $\pr$ of $M$  lies in the support of $Q$. However, for such $\pr$, we have $M_\pr=\hlc{M_\pr}0{\pr R_\pr}$.   Hence $N_\pr=M_\pr$ by \Cor{C:openval},  showing that $Q_\pr=0$. Suppose next that $M$ has no embedded primes and $\op{Supp}(Q)$ is nowhere dense in $\op{Supp}(M)$. Let $D:=\fcyc {}N+\fcyc{}Q-\fcyc {}M$ as given by   \Cor{C:semiaddcyc}, and let $\pr$ be an associated prime of $Q$. Since $\op{Supp}(Q)$ is nowhere dense, $\pr$ is not an associated prime of $M$. Since the cycle $\fcyc {}M-\fcyc {}N$ has   support in $\op{Ass}(M)$ and is equal to $ \fcyc {}Q-D$, it must be the zero cycle, since $\fcyc {}Q$ and $D$ have support disjoint from $\op{Ass}(M)$. Hence $N\sub M$ is open by \Cor{C:openval}.
\end{proof} 

It is useful to reformulate the above equivalence for modules without embedded primes: $N\sub M$ is open \iff\ $\op{Ass} (M)\cap\op{Ass}(M/N)=\emptyset$.
\Cor{C:openemb} together with \Prop{P:big}  yields:

\begin{corollary}\label{C:nillow}
If $N\sub M$ is open then   $\op{dim}(M/N)<\op{dim}(M)$, and the converse holds if $M$ is  unmixed.\qed
\end{corollary}  \begin{proposition}\label{P:dimopen}
Let $M$ be a $d$-dimensional module. The following are equivalent for a submodule $N\sub M$:
\begin{enumerate}
\item \label{i:subopen} $N$ is open in $M$;
\item\label{i:subsubopen} $\dim M/N<d$ and $\fl d_{d-1}(N)$ is open in $\fl d_{d-1}(M)$;
\item\label{i:subdimopen} $\dim{\fl d_i(M)/\fl d_i(N)}<i$,  for all $i$.
\end{enumerate}
\end{proposition} 
\begin{proof}
Since $\fl d_i(N)=\fl d_i(M)\cap N$, assumption~\eqref{i:subopen} implies the other two  by \Cor{C:nillow} and \Thm{T:cantop} below. Assume \eqref{i:subsubopen} holds. By the first of these two conditions and top additivity applied to $\Exactseq NM{M/N}$, we get $\mul Md=\mul Nd$. By
\Prop{P:dimfil}, the second condition yields $\len N^{<d}=\len M^{<d}$, and hence $\len N=\len M$, that is to say, $N$ is open. Finally, assuming \eqref{i:subdimopen}, let us prove by induction on $i$ that $\fl d_i(N)$ is open in $\fl d_i(M)$, where the case $i=d$ gives the desired result. If $i=0$, the assumption states that $\fl d_0(N)$ and $\fl d_0(M)$ are actually equal, and so we may assume $i>0$. Applying the first equivalence to the inclusion $\fl d_i(N)\sub\fl d_i(M)$ and using that by induction $\fl d_{i-1}(N)$ is open in $\fl d_{i-1}(M)$, we see that the former is open too.
\end{proof} 

In some cases, we have to check less:

\begin{proposition}\label{P:openid}
Let $I$ be an ideal in a $d$-dimensional ring $R$. If  $\dim{R/I}<d$ and $I$ contains an ideal $J$ such that $R/J$ is unmixed of dimension $d$, then $I$ is open.
\end{proposition} 
\begin{proof}
Let $\len R=e\omega^d\ssum\opord a$ with $\opord a$ of degree at most $d-1$. By assumption,  $\len {R/J}=b\omega^d$, for some non-zero $b$. Since $R/I$ has dimension less than $d$, top additivity (\Cor{C:topadd}) applied to $\Exactseq IR{R/I}$ yields $\mul Id=\mul Rd=e$. Hence $\len I=e\omega^d\ssum\opord b$, for some ordinal $\opord b$ of degree at most $d-1$. On the other hand, by \Cor{C:orddim}, the exact sequence $\Exactseq JR{R/J}$ is \add, so that $\len R=\len J\ssum b\omega^d$, and hence $\len J=(e-b)\omega^d\ssum\opord a$. By \Thm{T:submod}, we have 
$$
\len J=(e-b)\omega^d\ssum\opord a\aleq{}\len I=e\omega^d\ssum\opord b\aleq{} e\omega^d\ssum\opord a=\len R.
$$
 Therefore, $\opord a=\opord b$, showing that $I$ is open. 
\end{proof} 

\begin{remark}\label{R:openid}
The module version is proven exactly in the same way: given a module $M$ and submodules $K\sub N$ such that $M/K$ is unmixed of dimension $d:=\dim M$ and $M/N$ has dimension strictly less than $d$, then $N$ is open.
Note that   $M/K$ is unmixed \iff\  $\fl d_{d-1}(M)\sub K$, and so, an equivalent criterion is that $\dim {M/N}<d$ and $\fl  d_{d-1}(M)\sub N$, then $N$ is open. The converse may fail: in $R=\pol k{x,y,z}/(x^2,xy)$, the ideal $(xz,y)$ is open by \cite[Proposition 3.10]{SchBinEndo}, but it does not contain $\fl d_1(R)=(x)$.
\end{remark} 

\begin{corollary}\label{C:minidopen}
In an equidimensional ring $R$,  a prime ideal is open \iff\ it is  non-minimal. If, moreover, $\dim R>0$, then $R$ admits a composition series consisting entirely of height one ideals (and the zero ideal). 
\end{corollary} 
\begin{proof}
Let $d=\dim R$, let $\pr$ be a non-minimal prime ideal, and let $\primary$ be a minimal prime ideal contained in $\pr$. Hence $\len{R/\primary}=\omega^d$  and $\dim {R/\pr}<d$, so that $\pr$ is open by \Prop{P:openid}. On the other hand, $\primary$ is not open by \Cor{C:nillow}, proving the converse. 

Let $\pr$ be a height one prime. By what we just proved, $\len \pr=\len R$, and hence, there exists a composition series inside $\pr$ by \Thm{T:chain}, the non-zero members of which therefore have height one.
\end{proof}

We conclude with some applications of open submodules to morphisms (more will be proven in \S\ref{s:degrad} below):

\begin{corollary}\label{C:zerohom}
If $\op{dim}(M)<\order {}N$, then $\hom RMN=0$.
\end{corollary} 
\begin{proof}
%
Let $d<e$ be the respective dimension of $M$ and order of $N$, and let $x\in M$. Since $x$ has dimension at most $d$, so does $f(x)$. Hence $f(x)$ must be zero, since $\fl d_{e-1}(N)=0$ by \Prop{P:dimfil}.
\end{proof}

\begin{theorem}\label{T:restdom}
If $N$ is an open submodule of an unmixed module $M$, then two endomorphisms of $M$ are equal \iff\ their restriction to $N$ are equal.
\end{theorem} 
\begin{proof}
Since $N$ is open, $M/N$ has dimension strictly less than the dimension, whence the order, of $M$ by \Cor{C:nillow}. Hence $\hom {}{M/N}M=0$ by \Cor{C:zerohom}. From the canonical exact sequence
$$
0\to \hom{}{M/N}N \to\ndo M\to\hom{}NM
$$ 
we get an inclusion $\ndo M\to \hom{}NM$, from which our claim follows.
\end{proof}

\section{The canonical topology}

As the name indicates, there is an underlying topology. To prove this, we need:

\begin{theorem}\label{T:cantop}
The inverse image of an open submodule under a morphism is again open. In particular, if $U\sub  M$ is open, and $N\sub M$ is arbitrary, then $N\cap U$ is open in $N$.
\end{theorem} 
\begin{proof}
We start with proving the second assertion. Let ${\opord m}:=\len M$, ${\opord n}:=\len N$ and ${\opord a}:=\len{N\cap U}$. We have an exact sequence
$$
\Exactseq{N\cap U}{N\oplus U}{N+U}.
$$
Since $N+U$ is again open, semi-additivity yields $\len{N\oplus U}\leq {\opord a}\ssum{\opord m}$. Since the former is equal to ${\opord n}\ssum{\opord m}$, we get ${\opord n}\leq {\opord a}$, showing that $N\cap U$ is open in $N$. To prove the first assertion, let $f\colon M\to N$ be an arbitrary \homo\ and let $V\sub N$ be an open submodule. Let $G\sub M\oplus N$ be the graph of $f$ and let $p\colon G\to M$ be the projection onto the first coordinate. Since $M\oplus V$ is open in $M\oplus N$ by semi-additivity, $(M\oplus V)\cap G$ is open in $G$, by what we just proved. Since $p$ is an isomorphism, the image of $(M\oplus V)\cap G$ under $p$ is therefore open  in $M$. But this image is just $\inverse fV$, and so we are done.
\end{proof}

We can now define a topology on $M$ by letting the collection of open submodules be a basis of open neighborhoods of $0\in M$. This is indeed a basis since the intersection of finitely many opens is again open by \Thm{T:cantop}.  An arbitrary open in this topology is then a  (possibly infinite) union of cosets $x+U$ with $U\sub M$ open and $x\in M$. If a submodule $N$ is a union of cosets $x_i+U_i$ of open submodules $U_i\sub M$, then one such coset,  $x+U$ say, must contain $0$, so that $U\sub N$, and hence the inequalities $\len U\leq\len N\leq \len M$ are all equalities, showing that $N$ is indeed open in the previous sense. 

We call this the \emph{canonical topology} on $M$, and \Thm{T:cantop} shows that any \homo\ is continuous in the canonical topology. Moreover,  multiplication on any ring is continuous: given $a_1,a_2\in R$ and an open ideal $I$ such that $a_1a_2\in I$, let $J_i:=a_iR+ I$. Hence $J_i$ is an open neighborhood of $a_i$ and $J_1\cdot J_2\sub a_1a_2R+I=I$. If $M$ has dimension zero, then the canonical topology is trivial, since $M$ is then the only open submodule. 
The complement of an open module $N$ is the union of all cosets $a+N$ with $a\notin N$, and hence is also open.  In particular, an open module is also closed, and the quotient topology on $M/N$ is discrete, whence in general  different from the canonical topology. The zero module is closed \iff\ the intersection of all open submodules is $0$, that is to say, \iff\ the canonical topology is Haussdorf. 
%
%
\begin{corollary}\label{C:nonArt}
A module is non-Artinian \iff\ its canonical topology is non-trivial.
\end{corollary} 
\begin{proof}
One direction is immediate since an Artinian module has no proper open submodules. For the converse, we show, by induction on  $\len M$, that $M$ has a non-trivial, open submodule.   Assume first that ${\opord m}$ is a limit ordinal. Choose a  submodule $N$ of $M$ of \hdim\ $1$. By \eqref{eq:quot}, this means $\len{M/N}=1$, and so $N$ is open by \Prop{P:big}. 
Next, assume ${\opord m}={\opord n}\ssum 1$. Let $H$ be a submodule of \hdim\ ${\opord n}$, so that by \eqref{eq:quot} again, $\bar M:=M/H$ has length ${\opord n}$. By induction, we can find a proper open submodule of $\bar M$, that is to say, we can find $N\varsubsetneq M$ containing $H$ such that $\len{N/H}={\opord n}$. Semi-additivity applied to the inclusion $H\sub N$ yields ${\opord n}\ordsum \len H\leq\len N$. Since $\len H\neq 0$ and $\len N\leq {\opord n}\ssum 1$, we get equality, that is to say, $\len N={\opord n}\ssum 1$, whence $N$ is open.
\end{proof}

\begin{proposition}\label{P:clzero}
The closure of a submodule   $N\sub M$ is equal to $N+\fl d_0(M)$. In particular, a submodule is closed \iff\ $\mul N0=\mul M0$.
\end{proposition} 
\begin{proof}
Let $H:=\fl d_0(M)$, and let $W$ be an open submodule containing $N$. By \Thm{T:cantop}, the intersection $H\cap W$ is open in $H$, and since $H$ has finite length (so that its topology is trivial), we must have $H=W\cap H$, proving that $H$ lies in $W$. As this holds for all opens  $W$ containing $N$, the closure of $N$ contains $H$.

Using \Cor{C:orddim}, one easily shows that the quotient topology on $M/H$ is equal to the canonical topology. Therefore, to calculate the closure, we may divide out $H$, assume that $M$ has positive order, and we then need to show that $N$ is closed. Let $x\in M$ be any   element not in $N$  and let $\maxim$ be a maximal ideal containing $(N:x)$. Since $M/\maxim^kM$ has finite length, each $\maxim^kM$ is open by \Prop{P:big}, whence so is each $N+\maxim^kM$. If $x$ is contained in each of these, then it is contained in their intersection $W:=\cap(N+\maxim^kM)$. By Krull's Intersection theorem, there exists $a\in \maxim$ such that $(1+a)W\sub N$. In particular, $1+a\in(N:x)\sub \maxim$, contradiction. So $x$ lies outside some open $N+\maxim^kM$, and hence does not belong to the closure of $N$. 
The last assertion is now also clear by \Thm{T:cohrk}, since $\mul M0$ is  the length of $\fl d_0(M)$.
\end{proof}

\begin{corollary}\label{C:hauss}
    A module has a separated canonical topology \iff\ its order is positive \iff\ any submodule is closed.\qed
\end{corollary} 

\begin{example}\label{E:adictop}
If $(R,\maxim)$ is   local and $M$ has positive depth, then the canonical topology refines the $\maxim$-adic topology on $M$: indeed $\maxim^k M$ is then open by \Prop{P:big}. We already showed that this is no longer true in rings of positive depth in \Rem{R:lenprim}.
Also note that the canonical topology on a local domain is strictly finer than its adic topology, since all non-zero ideals are open.   

The ring $\pow k{x,y}/(x^2,xy)$, of length $\omega\ssum 1$, is not Haussdorf as the closure of the zero ideal is the ideal $(x)$ by \Prop{P:clzero}. It is the only closed, non-open ideal, since the closure of any ideal must contain $x$ whence is open when different from $(x)$. In particular, whereas the canonical topology is not Haussdorf,  the adic one is.
\end{example} 

Recall that a submodule $N\sub M$ is called \emph{essential} (or \emph{large}), if it intersects any non-zero submodule non-trivially.

\begin{corollary}\label{C:bigbg}
An open submodule is {essential}. In particular,  $0$ is a limit point of any non-zero submodule.
\end{corollary} 
\begin{proof}
Let $N\sub M$   be open and  $H\sub M$   arbitrary. Suppose $H\cap N=0$, so that $H\oplus N$ embeds  as a submodule of $M$. In particular, $\len H\ssum\len N\leq \len M$ by semi-additivity, forcing $\len H$, whence $H$, to be zero.

To prove the second assertion, let $N$ be non-zero and let $U$ be an open containing $0$. Since $U$ is the union of cosets of open submodules and contains $0$, it must contain at least one open submodule $W$. Since $W$ is essential by our first assertion, $W\cap N\neq0$.
\end{proof} 

The converse is false: in an Artinian local ring, the socle is essential, but it is clearly not open.  A less trivial example is given by the ideal $\pr:=(x,y)$ in the ring $R=\pol k{x,y,z}/\pr^2$, which is essential but not open. Indeed, $\fcyc {}R=3[\pr]$ and hence $R$ has length $3\omega$ by \Thm{T:cohrk}. Since $R/\pr$ is a one-dimensional domain, its length is $\omega$ by  \Thm{C:orddim}, and hence $\len\pr=2\omega$ by \Cor{C:orddim}. To see that $\pr$ is essential, we can use the following proposition with $S=\pol k{x,y}/\pr^2$ (so that $R=\pol Sz$). 
 
\begin{proposition}\label{P:essff}
Let $(S,\pr)$ be a local ring and $S\to R$ a   flat extension. Then $\pr R$ is an essential ideal of $R$.
\end{proposition} 
\begin{proof}
If $\pr R$ were not essential, we could find a non-zero $x\in R$ such that $\pr R\cap xR=0$. In particular, $x\pr=0$. By  flatness, $x\in\ann S\pr R\sub\pr R$, contradiction.
\end{proof}

\begin{theorem}\label{T:indtop}
If $M$ has no embedded primes, then the induced topology on a submodule $N\sub M$ is the same as the canonical topology on $N$.
\end{theorem} 
\begin{proof}
In view of \Thm{T:cantop}, we only have to show that if $W\sub N$ is open, then there exists $U\sub M$ open such that $W=U\cap N$. Let $U$ be maximal such that $U\cap N=W$. Suppose $U$ is not open, so that there exists an associated prime $\pr$ of $M$ with $\hlc{U_\pr}0{\pr R_\pr}\varsubsetneq \hlc{M_\pr}0{\pr R_\pr}$ by \Cor{C:openval}. In particular, we can find $x\in M\setminus U$ with $s\pr x\sub U$, for some $s\notin\pr$. By maximality, $(U+Rx)\cap N$ must contain an element $n$ not in $W$. Write $n=u+rx$ with $u\in U$ and $r\in R$. In particular, $s\pr n=s\pr u$ lies in $U\cap N=W$. In other words, we showed that $\hlc{N_\pr/W_\pr}0{\pr R_\pr}\neq 0$. In particular, $N_\pr\neq 0$, so that $\pr$ is a minimal prime of $N$. By \Lem{L:loccohgen}, we have $\mul W\pr  <\mul W\pr  +\mul  {W/N}\pr=\mul N\pr  $, so that $\len W<\len N$ by \Thm{T:cohrk}, contradiction.
\end{proof} 

I do not know whether this remains still true if there are embedded primes, but see \cite[Corollary 4.4]{SchCond}.

\section{Application I: Acyclicity}\label{s:acyc}
For the remainder of this paper $R$ is a $d$-dimensional Noetherian ring. Furthermore,   $M$, $N$, \dots are finitely generated modules over $R$, of length ${\opord m}$, ${\opord n}$, etc.
We start with reproving the observation of Vasconcelos \cite{VasFlat} that a surjective endomorphism on a Noetherian module must be an isomorphism (the usual proof uses the determinant trick; see for instance \cite[Theorem 2.4]{Mats}).

\begin{corollary}\label{C:Vasc}
Any surjective endomorphism  is an isomorphism.
\end{corollary}
\begin{proof}
Let $M\to M$ be a surjective endomorphism with kernel  $N$, so that we have an exact sequence $\Exactseq NMM$, and therefore, by \Thm{T:semadd}, an inequality $\len M\ordsum \len N\leq \len M$. By simple ordinal arithmetic, this implies $\len N=0$, whence $N=0$.
\end{proof}

\begin{remark}\label{R:Vasc}
Our argument in fact proves that any surjection between modules of the same length must be an isomorphism, or more generally, if $f\colon M\to N$ is an epimorphism and $\len  M\leq \len N$, then $f$ is an isomorphism and $\len M=\len N$. More generally, endomorphisms are non-expansive in the following sense (note that by taking $N=M$, we  recover  Vasconcelos' result):
\end{remark}

\begin{theorem}\label{T:endosupim}
Let $f$ be an endomorphism on a Noetherian module $M$ and let $N$ be a submodule such that $N\sub f(N)$. Then $f(N)= N$ and the restriction $\restrict fN$ is an automorphism of $N$.
\end{theorem} 
\begin{proof}
Since $N\sub f(N)$, we have $\len N\leq\len{f(N)}$. On the other hand, $f$ induces a surjection $N\to f(N)$, showing that $\len {f(N)}\leq \len N$. Hence $\len N=\len {f(N)}$, implying by \eqref{eq:quot} that the surjection $N\to f(N)$ is an isomorphism. Now, from $N\sub f(N)$, we get $f(N)\sub f^2(N)$, and applying the same argument to the submodule $f(N)$, shows that $f$ is injective on $f(N)$. Repeating in this way, we see that $f$ is injective on each $f^n(N)$. Suppose now that the inclusion $N\sub f(N)$ is strict, so that we can find $a\in N$ with $f(a)\notin N$. Suppose for some $n$, we have  $f^n(a)\in f^{n-1}(N)$, say, $f^n(a)=f^{n-1}(b)$ with $b\in N$. Since $f$ is injective on each $f^k(N)$, we get $b=f(a)$, contradiction. Hence the  chain $N\varsubsetneq f(N)\varsubsetneq f^2(N)\varsubsetneq\dots$ is  strictly ascending, contradicting Noetherianity. In conclusion, $N=f(N)$ and the assertion follows.
\end{proof}

\begin{corollary}\label{C:subim}
If $N$ is a homomorphic image of $M$ which contains a submodule isomorphic to $M$, then $M\iso N$.
\end{corollary} 
\begin{proof}
  Since $M\into N$, semi-additivity yields $\len M\leq\len N$. By \Rem{R:Vasc},  the epimorphism $M\onto N$ must then be an isomorphism.
\end{proof}

The following result generalizes Miyata's result \cite{Miy} as we do not need to assume that the given sequence is left exact.

\begin{theorem}\label{T:miyata}
An exact sequence $M\to N\to C\to 0$ is split exact \iff\   the (abstract) modules  $N$ and $ M\oplus C$ are isomorphic.
\end{theorem} 
\begin{proof} One direction is just the definition of split exact. 
Let $\bar M$ be the image of $M$ and apply \Thm{T:semadd} to  $\Exactseq {\bar M}NC$ to get $\len N\leq\len C\ssum\len {\bar M}$. On the other hand, $N\iso M\oplus C$ yields $\len N=\len M\ssum\len C$, whence $\len M\leq\len {\bar M}$. Since $\bar M$ is a homomorphic image of $M$, they must be isomorphic by \Rem{R:Vasc}. Hence, we showed $M\to N$ is injective.   At this point we could invoke  \cite{Miy}, but we can as easily give a direct proof of splitness as follows.  Given a finitely generated $R$-module $H$,  since $M\tensor H\iso (N\tensor H)\oplus (C\tensor H)$,   the same argument  applied to the tensored exact sequence 
$$
M\tensor H\to  N\tensor H\to C\tensor H\to 0,
$$
  gives the injectivity of  the first arrow. We therefore showed that $M\to N$ is pure, whence split by \cite[Theorem 7.14]{Mats}. 
\end{proof} 

\begin{remark}\label{R:striuli}
Independently,  Striuli  proved the same result in \cite{StrExt}. However, as she only proves exactness in the local case, she still needs to invoke Miyata's original result, whereas our proof stands alone.
\end{remark}

\begin{theorem}\label{T:nonsingsub}
Let $X$ be a non-singular variety over an \acf\ $k$. Then a closed subscheme $Y\sub X$ with ideal of definition $\mathcal I$ is non-singular \iff\ $\Omega_{X/k}\tensor\loc_Y$ is locally isomorphic to $\mathcal I/\mathcal I^2\oplus \Omega_{Y/k} $.
\end{theorem}
\begin{proof}
Since $X$ is non-singular, its module of differentials $\Omega_{X/k}$ is locally free (\cite[Theorem 8.15]{Hart}), whence so is $\Omega_{X/k}\tensor\loc_Y$, and therefore so is its direct summand $\Omega_{Y/k}$. Moreover, by \Thm{T:miyata}, the conormal sequence
$$
\mathcal I/\mathcal I^2\to \Omega_{X/k}\tensor\loc_Y\to \Omega_{Y/k}\to 0
$$
 is then split exact, and the result now follows from \cite[Theorem 8.17]{Hart}.
\end{proof}  

\begin{theorem}\label{T:retract}
Let $A$ be a finitely generated $R$-algebra,   $I\sub A$ an ideal, and $\bar A:=A/I$.  The closed immersion $\op{Spec}\bar A\sub \op{Spec}(A/I^2)$ is a retract over $R$ \iff\ we have an (abstract) isomorphism of $A$-modules
\begin{equation}\label{eq:condiff}
\Omega_{A/R}\big/I\Omega_{A/R}\iso \Omega_{\bar A/R}\oplus I/I^2.
\end{equation} 
\end{theorem} 
\begin{proof}
One direction is easy, and if \eqref{eq:condiff} holds, then the conormal sequence
$$
  I/  I^2\to \Omega_{A/R}\big/I\Omega_{A/R}\to \Omega_{\bar A/R}\to 0
$$
is split exact by \Thm{T:miyata}, so that the result follows from \cite[Proposition 16.12]{Eis}.
\end{proof} 

 Given a module $M$ and an element $f\in R$, we say that $f$ is a \emph{parameter} on $M$, if $M/fM$ is non-zero but has dimension strictly less than $M$. For a local ring $(R,\maxim)$, an element $f\in\maxim$ is a parameter \iff\ $\dim{R/fR}=\dim R-1$, but this does not necessarily hold in non-local rings, like $\pol{\pow kx}y$ with $f=xy-1$.

\begin{theorem}\label{T:lenpar}
 An element $f\in R$ is a parameter on $M$ \iff\ the dimension of $\ann Mf$ as a module is  less than  that of $  M$.
\end{theorem} 
Recall that $\ann Mf$ is the submodule of all $y\in M$ such that $fy=0$.
\begin{proof}
Let $d:=\dim M$. Replacing $R$ by $R/\ann {}M$, we may assume $d=\dim R$. Put $\bar M:=M/fM$ and consider the exact sequence $\Exactseq {fM}M{\bar M}$. It follows from top additivity  (\Cor{C:topadd}) that $\dim{\bar M}<d$, \iff\ $M$ and $fM$ have the same generic length, that is to say, 
\begin{equation}\label{eq:mulfM}
\mul Md=\mul{fM}d.
\end{equation} 
 Applying top-additivity instead to the exact sequence $\Exactseq{\ann{M}f}M{fM}$, we see  that \eqref{eq:mulfM} in turn  is equivalent with   $\mul{\ann{M}f}d=0$, meaning that   $\ann{M}f$ has  dimension at most $d-1$.
\end{proof}

\section{Application II: Degradation}\label{s:degrad}

By \emph{degradation}, we mean the effect that source and target of a morphism have on its kernel. 
We start with a general observation about kernels: given two $R$-modules $M$ and $N$, let us denote the subset of $\grass RM$ consisting of all $\op{ker}(f)$, where $f\in\hom RMN$ runs over all morphisms, by $\mathfrak{ker}_R(M,N)$.

\begin{theorem}\label{T:kerlen}
As a subset of $\grass RM$, the ordered set $\mathfrak{ker}_R(M,N)$ has finite length.
\end{theorem} 
\begin{proof}
Let $f\colon M\to N$ be a morphism, and let  ${\opord q}$ be the  length of its   image.   By \Thm{T:submod}, we have  ${\opord q}\aleq{}\len N$. In particular, there are at most  $2^{\val N}$ possibilities for  ${\opord q}$. I claim that if $g\colon M\to N$ is a second morphism and $\op{ker}(g)\varsubsetneq \op{ker}(f)$, then $\len{\op{Im}(g)}$ is strictly bigger than ${\opord q}$. From this claim it then follows that any chain in $\mathfrak{ker}_R(M,N)$ has length at most $2^{\val N}$. To prove the claim, we have $\hd{\op{ker}(g)}=\len{N/\op{ker}(g)}=\len{\op{Im}(g)}$, by \eqref{eq:quot}. By assumption, $\op{ker}(g)$ is strictly contained in $\op{ker}(f)$, and hence it has strictly bigger \hdim, showing the claim.
\end{proof} 


\begin{corollary}\label{C:uniker}
If $N$ is univalent (i.e., $\val N=1$), then there are no inclusion relations among the kernels of non-zero morphisms $M\to N$. In particular, for each $K\sub M$, the set $H_K$ of morphisms $M\to N$ with kernel equal to $K$ together with the zero morphism, is a submodule of $\hom RMN$.
\end{corollary}
\begin{proof}
By the  proof of \Thm{T:kerlen}, there are only two possibilities for the \hdim\ of a kernel, one of which is zero (given by the zero morphism). Assume $f,g\in H_K$. Since $K$ then lies in the kernel of $rf+sg$, for any $r,s\in R$, the latter kernel is either $K$ or $M$.
\end{proof}

For each $v$, let $\grass vM$ be the subset of the \sch\ $\gr M$ consisting of all submodules for which $M/N$ has valence at most $v$. The same argument shows that each $\grass vM$ has finite length: indeed, for $N\in\grass vM$, we have $\hd N\aleq{}v\omega^d\ssum v\omega^{d-1}\ssum \dots\ssum v$, where $d$ is the dimension of $M$, and therefore, we only have finitely many possibilities for $\hd N$. Note that the union of the $\grass vM$ is $\gr M$, so that the \sch\ can be written as a union of suborders of finite length.  
We can now list some examples of degradation:

\begin{corollary}\label{C:openker}
If $M$ and $N$ have no associated primes in common, then any morphism between them has open kernel.
\end{corollary} 
\begin{proof}
Let $f\colon M\to N$ be a morphism and let $K\sub M$ and $Q\sub N$ be its respective kernel and image, so that we have an exact sequence $\Exactseq KMQ$. For any associated prime $\pr$ of $M$, we have $\hlc{N_\pr}0{\pr R_\pr}=0$ whence also $\hlc{Q_\pr}0{\pr R_\pr}=0$. By left exactness of local cohomology,  $\hlc{K_\pr}0{\pr R_\pr}=\hlc{M_\pr}0{\pr R_\pr}$, and hence $K$ is open by \Cor{C:openval}.
\end{proof} 

\begin{corollary}\label{C:noembopen}
Let $M$ be a module without embedded primes. A submodule $U\sub M$ is open \iff\ every \homo\ $f\colon M\to M/U$ has open kernel.
\end{corollary} 
\begin{proof}
Sufficiency follows from the discussion following \Cor{C:openemb} and \Cor{C:openker}. As for necessity, this is immediate when applied to the canonical projection $M\to M/U$ whose kernel is precisely $U$.
\end{proof} 


Since endomorphism rings are in general non-commutative, the set of nilpotent elements is not necessarily an ideal, but we can always find a subcollection which is:

\begin{proposition}\label{P:openkerendo}
Given a module $M$, let $\openndo M$ be the subset   of all endomorphisms whose kernel is open. Then $\openndo M$ is a   two-sided ideal of $\ndo M$ consisting entirely of nilpotent endomorphisms.
\end{proposition} 
\begin{proof}
Suppose $a,b\in\mathfrak O:=\openndo M$ and $c\in\ndo M$, then the kernel of $a+b$ contains the open $\op{ker}(a)\cap\op{ker}(b)$ and hence is open itself. Likewise, the kernel of $ca$ contains the open $\op{ker}(a)$, and the kernel of $ac$ is the inverse image of $\op{ker}(a)$ under $c$, which is   open by continuity (\Thm{T:cantop}). This shows that $\mathfrak O$ is a two-sided ideal. 

Finally, some power $a^n$ satisfies $\op{ker}(a^n)=\op{ker}(a^{n+1})$ by Noetherianity, which implies $\op{ker}(a^n)\cap\op{Im}(a^n)=0$. Since $\op{ker}(a^n)$ is a fortiori open, whence essential by \Cor{C:bigbg}, the submodule $\op{Im}(a^n)$ must be zero, showing that $a$ is nilpotent.
\end{proof}

Note that the second part of the proof actually gives:

\begin{corollary}\label{C:opennil}
If an endomorphism has open kernel, then it is nilpotent.\qed
\end{corollary}

\begin{theorem}\label{T:pingpong} 
Suppose $M$ is torsion-free over $\zet$.
If $M$ and $N$ have no associated primes in common, then there exists $k\in\nat$, such that for any choice of morphisms $f_i\colon M\to N$ and $g_i\colon N\to M$, with $i=\range 1k$, the composition $g_kf_k\cdots g_1f_1=0$.
\end{theorem} 
\begin{proof}
By \Prop{P:openkerendo}, the ideal $\mathfrak O:=\openndo M$ of endomorphisms  with open kernel contains only nilpotent elements.    Hence, by   the Nagata-Higman Theorem (\cite{HigNag,NagNil}), the ideal $\mathfrak O$ is   nilpotent,   that is to say, $\mathfrak O^k=0$, for some $k$. Since the kernel of each $f_i$ is open  by \Cor{C:openker}, so is the kernel of each $g_if_i$, showing that $g_if_i\in\mathfrak O$, and the claim follows.
\end{proof} 

\begin{corollary}\label{C:dimpingpong}
Suppose $\mathbb Q\sub R$. 
Let $e$ be the maximal dimension of a common associated prime of $M$ and $N$. Then there is some $k$, such that for any choice of morphisms $f_i\colon M\to N$ and $g_i\colon N\to M$, with $i=\range 1k$, the image of the composition $g_kf_k\cdots g_1f_1$ has dimension at most $e$.
\end{corollary} 
\begin{proof}
We take the convention that $e=-1$ if there are no common associated primes and we assign $-1$ to the dimension of the zero module. Hence the assertion is now just \Thm{T:pingpong} in  case $e=-1$. For $e\geq0$, let $M':=M/\fl d_e(M)$ and $N':=N/\fl d_e(N)$. Since $M'$ and $N'$ have no associated primes in common by maximality of $e$, we can find some $k$   as in \Thm{T:pingpong}. Choose $k$ many morphisms $f_i$ and $g_i$ as in the hypothesis, and let $h$ be their composition. Since morphisms cannot increase dimension, they induce morphisms between $M'$ and $N'$, and hence, by choice of $k$, the endomorphism on $M'$ induced by $h$ is zero. It follows that $h(M)\sub\fl d_e(M)$, as claimed.
\end{proof}

\begin{remark}\label{R:torsion}
The torsion restrictions above and below come from our application of the Nagata-Higman Theorem, which requires some form of torsion-freeness (see \cite[Remark 6.5]{SchBinEndo} for a further discussion). One can weaken these assumptions: for instance, in \Cor{C:dimpingpong}, we only need that $\pr\cap \zet=0$, for any associated prime $\pr$ of $M$.  
\end{remark} 

To extend \Thm{T:pingpong} to  several modules, let us say that an endomorphism $f\in\ndo M$    \emph{reflects through} a collection of modules $\mathcal N$, if we can factor $f$ as $M\to N\to M$, for each $N\in\mathcal N$.  Of course, any endomorphism reflects through $M$ itself. We can now prove:

\begin{theorem}\label{T:multipingpong}
Let $\mathcal N$ be a collection  of $R$-modules such that   no prime ideal   is associated  to every $N\in\mathcal N$. For any  module $M\in\mathcal N$  without $\zet$-torsion,  there exists $k\in\nat$, so that any product of $k$-many endomorphisms on $M$  reflecting through $\mathcal N$ is zero.
\end{theorem}  
\begin{proof}
As in the proof of \Thm{T:pingpong}, it suffices  to show that any  endomorphism $f\in\ndo M$ reflecting to $\mathcal N$  has open kernel. Let $K$ be its kernel and let $\pr$ be an associated prime of $M$. By assumption, there exists $N\in\mathcal N$ such that $\pr$ is not an associated prime of $N$. By definition, there  exists a factorization $f=hg$ with $g\colon M\to N$ and $h\colon N\to M$. Let $H$ be the kernel of $g$. By the argument in the proof of \Cor{C:openker} applied to $g$, we get $\hlc{H_\pr}0{\pr R_\pr}=\hlc{M_\pr}0{\pr R_\pr}$. Since $H\sub K$, this implies $\hlc{K_\pr}0{\pr R_\pr}=\hlc{M_\pr}0{\pr R_\pr}$, and since this holds for all associated primes $\pr$ of $M$, we proved   by \Cor{C:openval} that $K$ is open. 
\end{proof} 

If, instead, there are common associated primes, let $e$ be the maximum of their dimensions.  By the same argument as in the proof of \Cor{C:dimpingpong}, we may then conclude that the image of any product of $k$-many  endomorphisms reflecting through $\mathcal N$  has dimension at most $e$.

\section{Appendix: shuffle sums}\label{s:ssum}

 Recall that (standard) addition on ordinals is not commutative. We  will give three different but equivalent ways of defining a different, commutative addition
operation on the class of ordinals, which we temporarily will denote as $\ssum$, $\bsum$ and
$\tsum$. The sum $\ssum$ is also known as the \emph{natural (Hessenberg)
sum} and is often denoted $\#$. Recall our convention for scalar multiplication on the left (see \S\ref{s:Ord}). Every ordinal ${\opord a}$---we no longer restrict to those of finite degree---can be written as a sum
\begin{equation}\label{eq:CNF}
{\opord a}=a_n\omega^{\opord n_n}\ordsum \dots\ordsum a_1\omega^{\opord n_1}
\end{equation}
where the $\opord n_i$ (called the \emph{exponents}) form a strictly ascending
chain of ordinals, that is to say, $\opord n_1<\dots<\opord n_n$, and the $a_i$ (called the
\emph{coefficients})
are non-negative integers. This decomposition (in base $\omega$) is unique if 
we moreover require that all coefficients $a_i$ are non-zero, called
the \emph{Cantor normal form} (\emph{in base $\omega$}) of ${\opord a}$. 
If \eqref{eq:CNF} is in Cantor normal form, then we call the highest (respectively, lowest) occurring exponent, the \emph{degree} (respectively,   the \emph{order}) of ${\opord a}$ and we denote these respectively by $\op{deg}({\opord a}):=\opord n_n$ and  $\low{\opord a}:=\opord n_1$.
Note that ${\opord a}$ is a successor ordinal \iff\
$\low{\opord a}=0$.  

Given a second
ordinal ${\opord b}$, we may assume that
after possibly adding some more exponents, that it can also be written in the
form~\eqref{eq:CNF}, with coefficients $b_i\geq 0$ instead of the $a_i$. We now
define 
$$
{\opord a}\ssum{\opord b}:=
(a_n+b_n)\omega^{\opord n_n}\ordsum \dots\ordsum (a_1+b_1)\omega^{\opord n_1}.
$$
It follows that ${\opord a}\ssum{\opord b}$ is equal to ${\opord b}\ssum{\opord a}$ and is
greater than or equal to both ${\opord a}\ordsum {\opord b}$ and   ${\opord b}\ordsum {\opord a}$. For
instance if ${\opord a}=\omega\ordsum 1$ then ${\opord a}\ssum{\opord a}=2\omega\ordsum 2$
whereas ${\opord a}\ordsum {\opord a}=2\omega\ordsum 1$. In case both
ordinals are finite, ${\opord a}\ssum{\opord b}={\opord a}\ordsum {\opord b}$. It is easy to check
that we have the following \emph{finite distributivity property}:
\begin{equation}\label{eq:findist}
({\opord a}\ssum{\opord b})\ordsum 1=({\opord a}\ordsum 1)\ssum{\opord b}={\opord a}\ssum({\opord b}\ordsum 1).
\end{equation}
In fact, this follows from the more general property that
$({\opord a}\ssum{\opord b})\ordsum {\opord q}=({\opord a}\ordsum {\opord q})\ssum{\opord b}=
{\opord a}\ssum({\opord b}\ordsum {\opord q})$ for all ${\opord q}<
\omega^{o+1}$, where $o$ is the minimum of $\low{\opord a}$ and
$\low{\opord b}$.

For the second definition, we use transfinite induction on the pairs 
$({\opord a},{\opord b})$ ordered
lexicographically.\footnote{As alluded to above, the counterintuitive notation normally adopted for the resulting ordinal is   ${\opord b}{\opord a}$.}
Define ${\opord a} \bsum0:={\opord a}$ and $0\bsum{\opord b}:={\opord b}$ so that
we may assume ${\opord a},{\opord b}>0$.  If ${\opord a}$ is a
successor ordinal with predecessor   ${\opord a}'$, then we define ${\opord a}\bsum{\opord b}$ as
$({\opord a}'\bsum{\opord b})\ordsum 1$. Similarly, if ${\opord b}$ is a successor
ordinal with predecessor   ${\opord b}'$, then we define ${\opord a}\bsum{\opord b}$ as $({\opord a}\bsum{\opord b}')\ordsum 1$.
Note that by transfinite induction,  both definitions agree when both ${\opord a}$ and ${\opord b}$ are successor
ordinals, so that we have no ambiguity in defining this sum operation when at
least one of the components is a successor ordinal. So remains the case that
both are limit ordinals. If  $\low{\opord a}\leq\low{\opord b}$, then we  let ${\opord a} \bsum {\opord b}$
be equal to the supremum of the ${\opord d} \bsum {\opord b}$ for all
${\opord d} <{\opord a}$. In the remaining case, when  $\low{\opord a}>\low{\opord b}$,   we  let ${\opord a} \bsum {\opord b}$
be equal to the supremum of the ${\opord a} \bsum {\opord d} $ for all
${\opord d} <{\opord b}$. This concludes the definition of $\bsum$.
 
 Finally, define
${\opord a} \tsum {\opord b}$ as
the supremum of all sums ${\opord a}_1\ordsum {\opord b}_1\ordsum \dots\ordsum {\opord a}_n\ordsum {\opord b}_n$, where the
supremum is taken  over all $n$ and all decompositions
${\opord a}={\opord a}_1\ordsum \dots\ordsum {\opord a}_n$ and ${\opord b}={\opord b}_1\ordsum \dots\ordsum {\opord b}_n$, with ${\opord a}_i,{\opord b}_i$ ordinals. Loosely speaking,
${\opord a} \tsum {\opord b}$ is the
largest possible ordering one can obtain by   \emph{shuffling} pieces of 
${\opord a}$ and ${\opord b}$. Since we may take
${\opord a}_1=0={\opord b}_n$, one checks that
${\opord a} \tsum {\opord b}={\opord b} \tsum {\opord a}$.

\begin{theorem}\label{T:ssum}
For all ordinals ${\opord a},{\opord b}$ we have
${\opord a}\ssum{\opord b}={\opord a}\bsum{\opord b}={\opord a}\tsum{\opord b}$.
\end{theorem}
\begin{proof}
Let ${\opord c}:={\opord a}\ssum{\opord b}$, $\bar{\opord c}:={\opord a} \bsum {\opord b}$ and
$\tilde{\opord c}:={\opord a} \tsum {\opord b}$. 
We first prove ${\opord c}=\bar{\opord c}$ by induction on  
$({\opord a},{\opord b})$ (in the lexicographical order). Since the case ${\opord a}=0$ or ${\opord b}=0$ is trivial,  we may take
${\opord a},{\opord b}>0$. If   ${\opord a}$ is a successor ordinal with predecessor   ${\opord a}'$, then
\begin{equation*}
\bar{\opord c}=({\opord a}' \bsum {\opord b})\ordsum 1=({\opord a}'\ssum{\opord b})\ordsum 1={\opord a}\ssum{\opord b}={\opord c},
\end{equation*} 
where the first equality is by definition, the second by
induction and the third by  the finite distributivity
property~\eqref{eq:findist}.  Replacing the role of ${\opord a}$ and ${\opord b}$, the 
same argument can be used to treat the case when ${\opord b}$ is a successor
ordinal. So we may assume that both are limit ordinals. There are again two 
cases, namely $\low{\opord a}\leq\low{\opord b}$ and $\low{\opord a}>\low{\opord b}$. By
symmetry, the
argument for the second case is similar as for the first, so we will only give
the details for the first case. Write ${\opord a}$ as ${\opord a}'\ordsum \omega^o$ where
$o:=\low{\opord a}$. 
By definition,
$\bar{\opord c}$ is the supremum of all ${\opord d}  \bsum {\opord b}$ with
${\opord d} <{\opord a}$. A cofinal subset of such ${\opord d} $
are the ones of the form ${\opord a}'\ordsum {\opord q}$  
with $0<{\opord q}<\omega^o$, so that $\bar{\opord c}$ is the
supremum of all $({\opord a}'\ordsum {\opord q}) \bsum {\opord b}$ for $0<{\opord q}<\omega^o$.
By induction, $\bar{\opord c}$ is the supremum of all 
\begin{equation}\label{eq:g'}
({\opord a}'\ordsum {\opord q})\ssum{\opord b}=({\opord a}'\ssum{\opord b}) \ordsum {\opord q},
\end{equation} 
where the equality holds because $o\leq \low{\opord b}$.
Taking the supremum of the ordinals in \eqref{eq:g'} for ${\opord q}<\omega^o$, we
get that $\bar{\opord c}=({\opord a}'\ssum{\opord b})\ordsum \omega^o$. Using the remark following
\eqref{eq:findist} one checks that this is
just $({\opord a}'\ordsum \omega^o)\ssum{\opord b}={\opord a}\ssum{\opord b}={\opord c}$. 

The inequality ${\opord c}\leq\tilde{\opord c}$ is clear using the shuffle of the
terms in the Cantor normal forms~\eqref{eq:CNF} for ${\opord a}$ and ${\opord b}$.
To finish the proof,
we therefore need to show, by induction on ${\opord a}$, that
\begin{equation}\label{eq:g3}
{\opord a}_1\ordsum {\opord b}_1\ordsum \dots\ordsum {\opord a}_n\ordsum {\opord b}_n\leq\bar{\opord c},
\end{equation}
for all decompositions ${\opord a}={\opord a}_1\ordsum \dots\ordsum {\opord a}_n$ and
${\opord b}={\opord b}_1\ordsum \dots\ordsum {\opord b}_n$. Since $\tsum$ is commutative, we may assume
$\low{\opord a}\leq\low{\opord b}$ and, moreover,   that ${\opord a}_n>0$.  Suppose
first that ${\opord a}$ is a successor ordinal with predecessor   ${\opord a}'$. In particular,
${\opord a}_n$ is also a successor ordinal, with predecessor, say,   ${\opord a}_n'$. By
definition, $\bar{\opord c}=({\opord a}' \bsum {\opord b})\ordsum 1$. Using
the decomposition ${\opord a}'={\opord a}_1\ordsum \dots\ordsum {\opord a}_{n-1}\ordsum {\opord a}_n'$ and
induction, we get that 
${\opord a}_1\ordsum {\opord b}_1\ordsum \dots\ordsum {\opord b}_n\ordsum {\opord a}_n'\leq {\opord a}' \bsum {\opord b}$. Taking
successors of both ordinals then yields \eqref{eq:g3}. Hence suppose ${\opord a}$
is a limit ordinal. Let ${\opord q}<{\opord a}_n$ and apply the induction to each 
${\opord d} :={\opord a}_1\ordsum \dots\ordsum {\opord a}_{n-1}\ordsum {\opord q}$, to get
$$
{\opord a}_1\ordsum {\opord b}_1\ordsum \dots\ordsum {\opord b}_{n-1}\ordsum {\opord q}\ordsum {\opord b}_n\leq {\opord d}  \bsum  {\opord b}.
$$
Taking suprema of both sides then yields inequality  \eqref{eq:g3}. 
\end{proof}

We will denote this new sum simply by $\ssum$ and refer to
it as the \emph{shuffle sum} of two ordinals, in view of its third equivalent
form.

\subsection*{Product Orders}\label{s:prod}
The \emph{product} of two  partially ordered sets $P$ and $Q$ is defined to be
the Cartesian product $P\times Q$ ordered by the rule $(a,b)\leq (a',b')$
\iff\ $a\leq a'$ and $b\leq b'$. 
The map $(a,b)\mapsto (b,a)$
is an order-preserving bijection between $P\times Q$ and $Q\times P$. It is
easy to check that  if both $P$ and $Q$ are partial well-orders, then so is $P\times
Q$. 
%
%

\begin{theorem}[Product Formula]\label{T:prod}
Given partial well-orders $P$ and $Q$, we have an equality $\len{P\times Q}= \len P\ssum\len Q$.
\end{theorem}
\begin{proof}
We prove the more general fact that   
\begin{equation}\label{eq:odimprod}
\hd{a,b}=\hd a\ssum\hd b 
\end{equation}
 for all
$a\in P$ and $b\in Q$, from which the assertion follows by taking suprema over
all elements in $P$ and $Q$. Note that we have not written superscripts   to denote on which ordered set
the \hdim\ is calculated since this is clear
from the context. To prove \eqref{eq:odimprod}, we may assume by
transfinite induction that it holds for all pairs $(a',b')$ strictly below $(a,b)$.
Put ${\opord a}:=\hd a$, ${\opord b}:=\hd b$ and ${\opord c}:=\hd{a,b}$. Since
$\hd{a,b}=\hd{b,a}$, via the isomorphism $P\times Q\iso Q\times P$, we may
assume that $\low{\opord a}\leq\low{\opord b}$ whenever this assumption is required
(namely, when dealing with limit ordinals). 
We
start with proving the
inequality ${\opord a}\ssum{\opord b}\leq{\opord c}$. If   ${\opord a}={\opord b}=0$ then  
$a$ and $b$ are minimal elements in respectively $P$ and $Q$, whence so is $(a,b)$ in $P\times Q$, that is to say,  ${\opord c}=0$. So we may
assume,
after perhaps exchanging $P$ with $Q$ that ${\opord a}>0$. Suppose ${\opord a}$ is a successor
ordinal  with predecessor ${\opord a}'$. Hence there exists $a'<a$  in $P$ 
with $\hd{a'}={\opord a}'$.
By induction, $\hd{a',b}={\opord a}'\ssum{\opord b}$ and hence 
${\opord c}=\hd{a,b}$ is at least $({\opord a}'\ssum{\opord b})\ssum 1={\opord a}\ssum{\opord b}$,
where the last equality follows from Theorem~\ref{T:ssum}.
If ${\opord a}$
is a limit ordinal, then there exists for each ${\opord d} <{\opord a}$ an element
$a_{\opord d} <a$ of \hdim\ ${\opord d} $. By
induction
$\hd{a_{\opord d} ,b}={\opord d} \ssum{\opord b}$ and hence $\hd{a,b}$ is at least
${\opord a}\ssum {\opord b}$ by  Theorem~\ref{T:ssum}.  This concludes the
proof that ${\opord a}\ssum{\opord b}\leq {\opord c}$. 

For the converse inequality, assume first that ${\opord c}$ is a
successor ordinal  with predecessor ${\opord c}'$. By definition, there exists $(a',b')<(a,b)$ in $P\times Q$
of \hdim\   ${\opord c}'$.
By induction, ${\opord c}'=\hd{a',b'}=\hd {a'}\ssum\hd{b'}$, from which we get ${\opord c}\leq {\opord a}\ssum{\opord b}$. A similar argument can be
used to treat the limit case and the details are left to the reader.
\end{proof}

%
%

%
%

\begin{thebibliography}{10}

\bibitem{AschPong}
M.~Aschenbrenner and W.Y. Pong, \emph{Orderings of monomial ideals}, Fund.
  Math. \textbf{181} (2004), no.~1, 27--74.

\bibitem{AtiMac}
M.~Atiyah and G.~Macdonald, \emph{Introduction to commutative algebra},
  Addison-Wesley Publishing Co., Reading, Mass., 1969.

\bibitem{BH}
W.~Bruns and J.~Herzog, \emph{Cohen-{M}acaulay rings}, Cambridge University
  Press, Cambridge, 1993.

\bibitem{Eis}
D.~Eisenbud, \emph{Commutative algebra with a view toward algebraic geometry},
  Graduate Texts in Mathematics, vol. 150, Springer-Verlag, New York, 1995.

\bibitem{GabCat}
Pierre Gabriel, \emph{Des cat\'egories ab\'eliennes}, Bull. Soc. Math. France
  \textbf{90} (1962), 323--448.

\bibitem{GabRen}
Pierre Gabriel and Rudolf Rentschler, \emph{Sur la dimension des anneaux et
  ensembles ordonn\'es}, C. R. Acad. Sci. Paris \textbf{265} (1967),
  A712--A715.

\bibitem{GarDec}
Steven Garavaglia, \emph{Decomposition of totally transcendental modules}, J.
  Symbolic Logic \textbf{45} (1980), no.~1, 155--164.

\bibitem{Hart}
R.~Hartshorne, \emph{Algebraic geometry}, Springer-Verlag, New York, 1977.

\bibitem{HigNag}
Graham Higman, \emph{On a conjecture of {N}agata}, Proc. Cambridge Philos. Soc.
  \textbf{52} (1956), 1--4.

\bibitem{KnuSur}
D.~E. Knuth, \emph{Surreal numbers}, Addison-Wesley Publishing Co., Reading,
  Mass.-London-Amsterdam, 1974.

\bibitem{Mats}
H.~Matsumura, \emph{Commutative ring theory}, Cambridge University Press,
  Cambridge, 1986.

\bibitem{McCRob}
J.~C. McConnell and J.~C. Robson, \emph{Noncommutative {N}oetherian rings},
  revised ed., Graduate Studies in Mathematics, vol.~30, American Mathematical
  Society, Providence, RI, 2001, With the cooperation of L. W. Small.

\bibitem{Miy}
Takehiko Miyata, \emph{Note on direct summands of modules}, J. Math. Kyoto
  Univ. \textbf{7} (1967), 65--69.

\bibitem{NagNil}
Masayoshi Nagata, \emph{On the nilpotency of nil-algebras}, J. Math. Soc. Japan
  \textbf{4} (1952), 296--301.

\bibitem{NasOys}
C.~N{\u{a}}st{\u{a}}sescu and F.~Van Oystaeyen, \emph{Dimensions of ring
  theory}, Mathematics and its Applications, vol.~36, D. Reidel Publishing Co.,
  Dordrecht, 1987.

\bibitem{ScheDimFil}
Peter Schenzel, \emph{On the dimension filtration and {C}ohen-{M}acaulay
  filtered modules}, Commutative algebra and algebraic geometry ({F}errara),
  Lecture Notes in Pure and Appl. Math., vol. 206, Dekker, New York, 1999,
  pp.~245--264.

\bibitem{SchBinEndo}
Hans Schoutens, \emph{Binary modules and their endomorphisms}, 2012.

\bibitem{SchCond}
\bysame, \emph{Condense modules and the {G}oldie dimension}, 2012.

\bibitem{SchUlBook}
\bysame, \emph{The use of ultraproducts in commutative algebra}, Lecture Notes
  in Mathematics, vol. 1999, Springer-Verlag, 2010.

\bibitem{Sim}
H.~Simmons, \emph{The {G}abriel dimension and {C}antor-{B}endixson rank of a
  ring}, Bull. London Math. Soc. \textbf{20} (1988), no.~1, 16--22.

\bibitem{StrExt}
Janet Striuli, \emph{On extensions of modules}, J. Algebra \textbf{285} (2005),
  no.~1, 383--398.

\bibitem{VasFlat}
W.~Vasconcelos, \emph{On finitely generated flat modules}, Trans. Amer. Math.
  Soc. \textbf{138} (1969), 505--512.

\end{thebibliography}
\providecommand{\bysame}{\leavevmode\hbox to3em{\hrulefill}\thinspace}
\providecommand{\MR}{\relax\ifhmode\unskip\space\fi MR }
\providecommand{\MRhref}[2]{%
  \href{http://www.ams.org/mathscinet-getitem?mr=#1}{#2}
}
\providecommand{\href}[2]{#2}

\end{document}